\documentclass[11pt,reqno]{amsart} %  please use amsart at 11pt
\usepackage{amsfonts}
\usepackage{mathrsfs}
\usepackage{comment}
\usepackage{lscape}
\usepackage{amssymb,latexsym}
\usepackage{cite} % to get Refs. [1,2,3] typeset as [1--3] automatically
\usepackage{rotating} % rotation
\usepackage[height=190mm,width=130mm]{geometry} % this is the journal  % text  area size
\usepackage{stackengine}
\setstackgap{S}{1pt}

\theoremstyle{plain}
\newtheorem{theorem}{Theorem}
\newtheorem{lemma}{Lemma}

\newtheorem{conclusion}{Conclusion}

\theoremstyle{definition}

\theoremstyle{remark}
\newtheorem{remark}{Remark}

\numberwithin{equation}{section} % to get equations numbered
\def\mbf#1{\mbox{\boldmath$#1$}}

\begin{document}
	\title[
	\hspace{-3mm} Linear (in)dependence of sequences of derivatives of $x^n\sin x$, $x^n\cos x$]
	{On the linear (in)dependence of sequences of derivatives of the functions $x^n\sin x$ and $x^n\cos x$}
	\author{Jozef Fecenko, Enno Diekema}
	\address{}
	\email{fecenko.euba@gmail.com, e.diekema@gmail.com}
	
	\thanks{}
	
\begin{abstract}
	The main goal of the paper is to prove that the sequence of functions $f(x), Df(x), \dots, D^{2n+1}f(x)$, where $f(x)$ is $x^n\sin x$ or $x^n\cos x$ are linearly independent. Or more generally: that the sequence of functions
	$D^kf(x), D^{k+1}f(x), \dots, D^{2n+k+1}f(x)$, $k\in \mathbb{N}$ is linearly independent.
	The problem is solved by a suitable transformation of the matrix of determinant of the Wronskian. Another approach for a special sequence of derivatives of functions uses only the definition of linear independence of functions. This approach generates interesting, non-elementary combinatorial identities.
\end{abstract}
	
	
\maketitle

\section{Introduction}
The $n$ functions \(f_1(x), f_1(x),\ldots, f_n(x)\) are linearly dependent if, for some 
\(c_1, c_2, \ldots, c_n \in \mathbb{R}\) not all zero, 
\[\sum_{i=1}^nc_if_i(x)=0\]
for all \(x\) in some interval \(I\). If the functions are not linearly dependent, they are said to be linearly independent. \cite{WE1}\par
The Wronskian of a set of \(n\) functions \(\phi_1, \phi_2, \ldots, \phi_n\) is defined by 
\[W(\phi_1, \phi_2, \ldots, \phi_n)=
\left|\begin{array}{cccc}
	\phi_1 & \phi_2 & \dots & \phi_n\\
	\phi_1^\prime & \phi_2^\prime & \dots & \phi_n^\prime\\
	\vdots & \vdots & \ddots & \vdots \\
	\phi_1^{(n-1)} & \phi_2^{(n-1)} & \dots & \phi_n^{(n-1)}
\end{array}\right|.\]
If the Wronskian is nonzero in some region, the functions \(\phi_i\) are linearly independent. If \(W=0\) over some range, the functions are linearly dependent somewhere in the range. \cite{WE2}  
\section{A method using the Wronskian}
The Wronskian, usually introduced in standard courses in Ordinary Differential Equations, is a very useful tool in algebraic geometry to detect ramification loci of
linear systems. \cite{GSC}
\begin{lemma}\label{defrk}
Let $\mbf {A}=[a_{ij}]_{n\times n}$ and 
\[\mbf{R}_k=[r_{k_{ij}}]_{n\times n}, \quad 
r_{k_{i,j}}=\left\{ \begin{array}{cl}
	1 & \textrm{if $i=j$ or ($i=j+1$ and $k+1\le i\le n$}) \\[0.5ex]
	0 & \textrm{otherwise} \end{array} \right.\]
\(k=1,2,\dots,n-1, n=2,3,\dots\)
then
\begin{equation}\label{eq:rtran}
\mbf{R}_k\mbf{A}=\begin{bmatrix}
a_{11}&a_{12}&\dots & a_{1n}\\
a_{21}& a_{22}&\dots & a_{2n}\\
\vdots & \vdots & \ddots & \vdots\\
a_{k1}&a_{k2}&\dots & a_{kn}\\
a_{k+1,1}+a_{k1} & a_{k+1,2}+a_{k2} & \dots & a_{k+1,n}+a_{kn}\\
a_{k+2,1}+a_{k+1,1} & a_{k+2,2}+a_{k+1,2} & \dots & a_{k+2,n}+a_{k+1,n}\\
\vdots & \vdots & \ddots & \vdots\\
a_{n,1}+a_{n-1,1} & a_{n,2}+a_{n-1,2} & \dots & a_{n,n}+a_{n-1,n}
\end{bmatrix}
\end{equation}
\begin{equation}\label{eq:stran}
\mbf{A}\mbf{R}_k^T=\begin{bmatrix}
a_{11}  & \dots & a_{1k} & a_{1,k+1}+a_{1k} &  \dots & a_{1n}+a_{1,n-1}\\
a_{21} & \dots & a_{2k} & a_{2,k+1}+a_{2k} &  \dots & a_{2n}+a_{2,n-1}\\
\vdots & \ddots & \vdots & \vdots &  \ddots & \vdots \\
a_{n-1,1} & \dots & a_{n-1,k} & a_{n-1,k+1}+a_{n-1,k} & \dots & a_{n-1,n}+a_{n-1,n-1}\\
a_{n1} & \dots & a_{nk} & a_{n,k+1}+a_{nk} &  \dots & a_{nn}+a_{n,n-1}
\end{bmatrix}
\end{equation}
 \(\mbf{R}_k^T\)  is the transposed matrix to the matrix \(\mbf{R}_k\).
\end{lemma}
The proof follows from the definition of matrix multiplication.
\begin{remark}
Of course for determinants, we have
\[\left| \mbf{R}_k\right|=\left|\mbf{R}_k^T \right|=1 \]
 $k=1,2,\dots,n-1$.
\end{remark}
\begin{lemma}{\label{lmm1}} Let $D$ be a differential operator and $f$ be a $4r$ ($r\in \mathbb{N}$) times differentiable function on an interval $I$. Let us denote the matrix 
	\begin{gather}
		\mbf{Z}=\left[ \begin{array}{cccccc}
			\scriptstyle  D^0f &\scriptstyle  D^2f &\scriptstyle  D^4f &\dots& \scriptstyle D^{2r-2}f &\scriptstyle  D^{2r}f\\
			\scriptstyle D^2f&\scriptstyle D^4f&\scriptstyle D^6f  & \dots &\scriptstyle D^{2r}f  &\scriptstyle  D^{2r+2}f \\
			\scriptstyle D^4f&\scriptstyle D^6f&\scriptstyle D^8f & \dots &\scriptstyle D^{2r+2}f&\scriptstyle D^{2r+4}f\\
			\scriptstyle \vdots &\scriptstyle \vdots  & \scriptstyle \vdots & \scriptstyle\ddots &\scriptstyle \vdots & \scriptstyle \vdots\\
			\scriptstyle D^{2m}f & \scriptstyle D^{2m+2}f& \scriptstyle D^{2m+4}f& \scriptstyle \dots & \scriptstyle D^{2(m+r)-2}f&\scriptstyle D^{2(m+r)}f\\
			\scriptstyle \vdots &\scriptstyle \vdots  & \scriptstyle \vdots & \scriptstyle\ddots &\scriptstyle \vdots & \scriptstyle \vdots\\
			\scriptstyle D^{2r-2}f & \scriptstyle D^{2r}f& \scriptstyle D^{2r+2}f& \scriptstyle \dots & \scriptstyle D^{4r-4}f&\scriptstyle D^{4r-2}f\\
			\scriptstyle D^{2r}f & \scriptstyle D^{2r+2}f& \scriptstyle D^{2r+4}f& \scriptstyle \dots & \scriptstyle D^{4r-2}f&\scriptstyle D^{4r}f
		\end{array}\label{eq:ASD}\right]
	\end{gather}
then 
\begin{gather}
		\left|\mbf{Z}\right|=\left|\arraycolsep=0.15pt\def\arraystretch{1.2}
		\begin{array}{cccccc}
			\scriptstyle  D^0f &\scriptstyle  (D^2+1)f &\scriptstyle  (D^2+1)^2f &\dots & \scriptstyle (D^2+1)^{r-1}f &\scriptstyle  (D^2+1)^rf\\
			\scriptstyle  (D^2+1)f &\scriptstyle  (D^2+1)^2f &\scriptstyle  (D^2+1)^3f &\dots& \scriptstyle (D^2+1)^{r}f &\scriptstyle  (D^2+1)^{r+1}f\\
			\scriptstyle  (D^2+1)^2f &\scriptstyle  (D^2+1)^3f &\scriptstyle  (D^2+1)^4f &\dots& \scriptstyle (D^2+1)^{r+1}f &\scriptstyle  (D^2+1)^{r+2}f\\
			\scriptstyle \vdots &\scriptstyle \vdots  & \scriptstyle \vdots & \scriptstyle\ddots &\scriptstyle \vdots & \scriptstyle \vdots\\
			\scriptstyle  (D^2+1)^{m}f &\scriptstyle  (D^2+1)^{m+1}f &\scriptstyle  (D^2+1)^{m+2}f &\dots& \scriptstyle (D^2+1)^{m+r-1}f &\scriptstyle  (D^2+1)^{m+r}f\\
			\scriptstyle \vdots &\scriptstyle \vdots  & \scriptstyle \vdots & \scriptstyle\ddots &\scriptstyle \vdots & \scriptstyle \vdots\\
			\scriptstyle  (D^2+1)^{r-1}f &\scriptstyle  (D^2+1)^{r}f &\scriptstyle  (D^2+1)^{r+1}f &\dots& \scriptstyle (D^2+1)^{2r-2}f &\scriptstyle  (D^2+1)^{2r-1}f\\
			\scriptstyle  (D^2+1)^{r}f &\scriptstyle  (D^2+1)^{r+1}f &\scriptstyle  (D^2+1)^{r+2}f &\dots& \scriptstyle (D^2+1)^{2r-1}f &\scriptstyle  (D^2+1)^{2r}f
		\end{array}\right|\label{eq:ase}
\end{gather}
\begin{proof}
In the rest of the paper we will simply use $1$ instead of $D^0$. \par 
It seems as if we have transformed the simpler form of the determinant into a more complicated one. However, it can sometimes have advantages.
The procedure that transforms the original matrix in (\ref{eq:ASD}) into the matrix of determinant in (\ref{eq:ase}) can be divided into two parts. 
In the first part, we will use  a sequence of elementary row operations:
adding a multiple of one row to another causes the determinant to remain the same. The elementary row operation we will do using matrix transformations $\mbf{R}_k$.
In the second part we will do the same on columns with the matrix transformation $\mbf {S}_k$\par 
The first part of the proof.\par
The first step: 

\begin{gather}
\left|\mbf{R}_1\mbf{Z}\right| =\notag\\
\left|\arraycolsep=0.15pt\def\arraystretch{1.2}
\begin{array}{cccccc}
	\scriptstyle  f &\scriptstyle  D^2f &\scriptstyle  D^4f &\scriptstyle\dots& \scriptstyle D^{2r-2}f &\scriptstyle  D^{2r}f\\
	\scriptstyle (D^2+1)f&\scriptstyle D^2(D^2+1)f&\scriptstyle D^4(D^2+1)f  &\scriptstyle \dots &\scriptstyle D^{2r-2}(D^2+1)f  &\scriptstyle  D^{2r}(D^2+1)f \\
	\scriptstyle D^2(D^2+1)f&\scriptstyle D^4(D^2+1)f&\scriptstyle D^6(D^2+1)f & \scriptstyle\dots &\scriptstyle D^{2r}(D^2+1)f&\scriptstyle D^{2r+2}(D^2+1)f\\
	\scriptstyle \vdots &\scriptstyle \vdots  & \scriptstyle \vdots & \scriptstyle\ddots &\scriptstyle \vdots & \scriptstyle \vdots\\
	\scriptstyle D^{2m-2}(D^2+1)f & \scriptstyle D^{2m}(D^2+1)f& \scriptstyle D^{2m+2}(D^2+1)f& \scriptstyle \dots & \scriptstyle D^{2(m+r)-4}(D^2+1)f&\scriptstyle D^{2(m+r)-2}(D^2+1)f\\
	\scriptstyle \vdots &\scriptstyle \vdots  & \scriptstyle \vdots & \scriptstyle\ddots &\scriptstyle \vdots & \scriptstyle \vdots\\
	\scriptstyle D^{2r-4}(D^2+1)f & \scriptstyle D^{2r-2}(D^2+1)f& \scriptstyle D^{2r}(D^2+1)f& \scriptstyle \dots & \scriptstyle D^{4r-6}(D^2+1)f&\scriptstyle D^{4r-4}(D^2+1)f\\
	\scriptstyle D^{2r-2}(D^2+1)f & \scriptstyle D^{2r}(D^2+1)f& \scriptstyle D^{2r+2}(D^2+1)f& \scriptstyle \dots & \scriptstyle D^{4r-4}(D^2+1)f&\scriptstyle D^{4r-2}(D^2+1)f
\end{array}\right|\label{eq:1rdet}
\end{gather}
The second step: 
\begin{gather}
\left|\mbf{R}_2\mbf{R}_1\mbf{Z}\right| =\notag\\
\left|\arraycolsep=0.15pt\def\arraystretch{1.2}\begin{array}{cccccc}
		\scriptstyle  f &\scriptstyle  D^2f &\scriptstyle  D^4f &\scriptstyle\dots& \scriptstyle D^{2r-2}f &\scriptstyle  D^{2r}f\\
		\scriptstyle (D^2+1)f&\scriptstyle D^2(D^2+1)f&\scriptstyle D^4(D^2+1)f  &\scriptstyle \dots &\scriptstyle D^{2r-2}(D^2+1)f  &\scriptstyle  D^{2r}(D^2+1)f \\
		\scriptstyle (D^2+1)^2f&\scriptstyle D^2(D^2+1)^2f&\scriptstyle D^4(D^2+1)^2f & \scriptstyle\dots &\scriptstyle D^{2r-2}(D^2+1)^2f&\scriptstyle D^{2r}(D^2+1)^2f\\
		\scriptstyle \vdots &\scriptstyle \vdots  & \scriptstyle \vdots & \scriptstyle\ddots &\scriptstyle \vdots & \scriptstyle \vdots\\
		\scriptstyle D^{2m-4}(D^2+1)^2f & \scriptstyle D^{2m-2}(D^2+1)^2f& \scriptstyle D^{2m}(D^2+1)^2f& \scriptstyle \dots & \scriptstyle D^{2(m+r)-6}(D^2+1)^2f&\scriptstyle D^{2(m+r)-4}(D^2+1)^2f\\
		\scriptstyle \vdots &\scriptstyle \vdots  & \scriptstyle \vdots & \scriptstyle\ddots &\scriptstyle \vdots & \scriptstyle \vdots\\
		\scriptstyle D^{2r-6}(D^2+1)^2f & \scriptstyle D^{2r-4}(D^2+1)^2f& \scriptstyle D^{2r-2}(D^2+1)^2f& \scriptstyle \dots & \scriptstyle D^{4r-8}(D^2+1)^2f&\scriptstyle D^{4r-6}(D^2+1)^2f\\
		\scriptstyle D^{2r-4}(D^2+1)^2f & \scriptstyle D^{2r-2}(D^2+1)^2f& \scriptstyle D^{2r}(D^2+1)^2f& \scriptstyle \dots & \scriptstyle D^{4r-6}(D^2+1)^2f & \scriptstyle D^{4r-4}(D^2+1)^2f
	\end{array}\right|\label{eq:2rdet}
\end{gather}
Using mathematical induction it can be proved that the $k$-th ($k=1,2,\dots,r$) step will be 
\begin{gather}
\left|\mbf{R}_k\dots\mbf{R}_2\mbf{R}_1\mbf{Z}\right| =\notag\\
\left|\arraycolsep=0.15pt\def\arraystretch{1.2}\begin{array}{ccccc}
      \scriptstyle  f &\scriptstyle  D^2f &\scriptstyle\dots& \scriptstyle D^{2r-2}f &\scriptstyle  D^{2r}f\\
\scriptstyle (D^2+1)f&\scriptstyle D^2(D^2+1)f&
\scriptstyle \dots &\scriptstyle D^{2r-2}(D^2+1)f  &\scriptstyle  D^{2r}(D^2+1)f \\
\scriptstyle (D^2+1)^2f&\scriptstyle D^2(D^2+1)^2f&\scriptstyle\dots &\scriptstyle D^{2r-2}(D^2+1)^2f&\scriptstyle D^{2r}(D^2+1)^2f\\
\scriptstyle \vdots &\scriptstyle \vdots  & \scriptstyle\ddots &\scriptstyle \vdots & \scriptstyle \vdots\\
\scriptstyle (D^2+1)^kf&\scriptstyle D^2(D^2+1)^kf&\scriptstyle\dots &\scriptstyle D^{2r-2}(D^2+1)^kf&\scriptstyle D^{2r}(D^2+1)^kf\\
\scriptstyle D^2(D^2+1)^kf&\scriptstyle D^4(D^2+1)^kf& \scriptstyle\dots &\scriptstyle D^{2r}(D^2+1)^kf&\scriptstyle D^{2r+2}(D^2+1)^kf\\
\scriptstyle \vdots &\scriptstyle \vdots  & \scriptstyle\ddots &\scriptstyle \vdots & \scriptstyle \vdots\\
\scriptstyle D^{2(m-k)}(D^2+1)^kf \!&\! \scriptstyle D^{2(m-k)+2}(D^2+1)^kf& 
\scriptstyle \dots & \scriptstyle D^{2(m+r-k)-2}(D^2+1)^kf& \scriptstyle D^{2(m+r-k)}(D^2+1)^kf\\
\scriptstyle \vdots &\scriptstyle \vdots  & \scriptstyle\ddots &\scriptstyle \vdots & \scriptstyle \vdots\\
\scriptstyle D^{2(r-k)-2}(D^2+1)^kf \!&\!\!\!\! \scriptstyle D^{2(r-k)}(D^2+1)^kf& 
\scriptstyle \dots & \scriptstyle D^{4r-2k-4}(D^2+1)^kf& \scriptstyle D^{4r-2k-2}(D^2+1)^kf\\
\scriptstyle D^{2(r-k)}(D^2+1)^kf \!&\! \scriptstyle D^{2(r-k)+2}(D^2+1)^kf& 
\scriptstyle \dots & \scriptstyle D^{4r-2k-2}(D^2+1)^kf& \scriptstyle D^{4r-2k}(D^2+1)^kf\\
\end{array}\right|\label{eq:krdet}
\end{gather}
Finally, if $k=r$ we get the final result of the first part of the proof.  
\begin{gather}
\left|\mbf{R}_r\dots\mbf{R}_2\mbf{R}_1\mbf{Z}\right| =\notag\\
\left|\arraycolsep=0.3pt\def\arraystretch{1.2}\begin{array}{cccccc}
	\scriptstyle  f &\scriptstyle  D^2f &\scriptstyle  D^4f &\scriptstyle\dots & \scriptstyle D^{2r-2}f &\scriptstyle  D^{2r}f\\
	\scriptstyle  (D^2+1)f &\scriptstyle  D^2(D^2+1)f &\scriptstyle  D^4(D^2+1)f &\scriptstyle\dots& \scriptstyle D^{2r-2}(D^2+1)f &\scriptstyle  D^{2r}(D^2+1)f\\
	\scriptstyle  (D^2+1)^2f &\scriptstyle  D^2(D^2+1)^2f &\scriptstyle  D^4(D^2+1)^2f &\scriptstyle\dots& \scriptstyle D^{2r-2}(D^2+1)^{2}f &\scriptstyle  D^{2r}(D^2+1)^{2}f\\
	\scriptstyle \vdots &\scriptstyle \vdots  & \scriptstyle \vdots & \scriptstyle\ddots &\scriptstyle \vdots & \scriptstyle \vdots\\
	\scriptstyle  (D^2+1)^{m}f &\scriptstyle D^2(D^2+1)^{m}f &\scriptstyle  D^4(D^2+1)^{m}f &\scriptstyle\dots& \scriptstyle D^{2r-2}(D^2+1)^{m}f &\scriptstyle  D^{2r}(D^2+1)^{m}f\\
	\scriptstyle \vdots &\scriptstyle \vdots  & \scriptstyle \vdots & \scriptstyle\ddots &\scriptstyle \vdots & \scriptstyle \vdots\\
	\scriptstyle  (D^2+1)^{r-1}f &\scriptstyle  D^2(D^2+1)^{r-1}f &\scriptstyle  D^4(D^2+1)^{r-1}f &\scriptstyle\dots& \scriptstyle D^{2r-2}(D^2+1)^{r-1}f &\scriptstyle D^{2r}(D^2+1)^{r-1}f\\
	\scriptstyle  (D^2+1)^{r}f &\scriptstyle  D^2(D^2+1)^{r}f &\scriptstyle  D^4(D^2+1)^{r}f &\scriptstyle\dots& \scriptstyle D^{2r-2}(D^2+1)^{r}f &\scriptstyle  D^{2r}(D^2+1)^{r}f
\end{array}\right|\label{eq:fin1prdet}
\end{gather}

The second part of the proof.\par
The first step: 
\begin{gather}
\left|\mbf{R}_r\dots\mbf{R}_2\mbf{R}_1\mbf{Z}\mbf{R}_1^T\right| =\notag\\
\left|\arraycolsep=0.3pt\def\arraystretch{1.2}\begin{array}{cccccc}
\scriptstyle  f &\scriptstyle  (D^2+1)f &\scriptstyle  D^2(D^2+1)f &\scriptstyle\dots & \scriptstyle D^{2r-4}(D^2+1)f &\scriptstyle  D^{2r-2}(D^2+1)f\\
\scriptstyle  (D^2+1)f &\scriptstyle  (D^2+1)^2f &\scriptstyle  D^2(D^2+1)^2f &\scriptstyle\dots& \scriptstyle D^{2r-4}(D^2+1)^2f &\scriptstyle  D^{2r-2}(D^2+1)^2f\\
\scriptstyle  (D^2+1)^2f &\scriptstyle  (D^2+1)^3f &\scriptstyle  D^2(D^2+1)^3f &\scriptstyle\dots& \scriptstyle D^{2r-4}(D^2+1)^{3}f &\scriptstyle  D^{2r-2}(D^2+1)^{3}f\\
\scriptstyle \vdots &\scriptstyle \vdots  & \scriptstyle \vdots & \scriptstyle\ddots &\scriptstyle \vdots & \scriptstyle \vdots\\
\scriptstyle  (D^2+1)^{m}f &\scriptstyle (D^2+1)^{m+1}f &\scriptstyle  D^2(D^2+1)^{m+1}f &\scriptstyle\dots& \scriptstyle D^{2r-4}(D^2+1)^{m+1}f &\scriptstyle  D^{2r-2}(D^2+1)^{m+1}f\\
\scriptstyle \vdots &\scriptstyle \vdots  & \scriptstyle \vdots & \scriptstyle\ddots &\scriptstyle \vdots & \scriptstyle \vdots\\
\scriptstyle  (D^2+1)^{r-1}f &\scriptstyle  (D^2+1)^{r}f &\scriptstyle  D^2(D^2+1)^{r}f &\scriptstyle\dots& \scriptstyle D^{2r-4}(D^2+1)^{r}f &\scriptstyle D^{2r-2}(D^2+1)^{r}f\\
\scriptstyle  (D^2+1)^{r}f &\scriptstyle  (D^2+1)^{r+1}f &\scriptstyle D^2(D^2+1)^{r+1}f &\scriptstyle\dots& \scriptstyle D^{2r-4}(D^2+1)^{r+1}f &\scriptstyle  D^{2r-2}(D^2+1)^{r+1}f
\end{array}\right|\label{eq:1sdet}
\end{gather}
The second step:
\begin{gather}
\left|\mbf{R}_r\dots\mbf{R}_2\mbf{R}_1\mbf{Z}\mbf{R}_1^T\mbf{R}_2^T\right| =\notag\\
\left|\arraycolsep=0.3pt\def\arraystretch{1.2}\begin{array}{cccccc}
		\scriptstyle  f &\scriptstyle  (D^2+1)f &\scriptstyle  (D^2+1)^2f &\scriptstyle\dots & \scriptstyle D^{2r-6}(D^2+1)^2f &\scriptstyle  D^{2r-4}(D^2+1)^2f\\
		\scriptstyle  (D^2+1)f &\scriptstyle  (D^2+1)^2f &\scriptstyle  (D^2+1)^3f &\scriptstyle\dots& \scriptstyle D^{2r-6}(D^2+1)^3f &\scriptstyle  D^{2r-4}(D^2+1)^3f\\
		\scriptstyle  (D^2+1)^2f &\scriptstyle  (D^2+1)^3f &\scriptstyle  (D^2+1)^4f &\scriptstyle\dots& \scriptstyle D^{2r-6}(D^2+1)^{4}f &\scriptstyle  D^{2r-4}(D^2+1)^{4}f\\
		\scriptstyle \vdots &\scriptstyle \vdots  & \scriptstyle \vdots & \scriptstyle\ddots &\scriptstyle \vdots & \scriptstyle \vdots\\
		\scriptstyle  (D^2+1)^{m}f &\scriptstyle (D^2+1)^{m+1}f &\scriptstyle  (D^2+1)^{m+2}f &\scriptstyle\dots& \scriptstyle D^{2r-6}(D^2+1)^{m+2}f &\scriptstyle  D^{2r-4}(D^2+1)^{m+2}f\\
		\scriptstyle \vdots &\scriptstyle \vdots  & \scriptstyle \vdots & \scriptstyle\ddots &\scriptstyle \vdots & \scriptstyle \vdots\\
		\scriptstyle  (D^2+1)^{r-1}f &\scriptstyle  (D^2+1)^{r}f &\scriptstyle  (D^2+1)^{r+1}f &\scriptstyle\dots& \scriptstyle D^{2r-6}(D^2+1)^{r+1}f &\scriptstyle D^{2r-4}(D^2+1)^{r+1}f\\
		\scriptstyle  (D^2+1)^{r}f &\scriptstyle  (D^2+1)^{r+1}f &\scriptstyle (D^2+1)^{r+2}f &\scriptstyle\dots& \scriptstyle D^{2r-6}(D^2+1)^{r+2}f &\scriptstyle  D^{2r-4}(D^2+1)^{r+2}f
	\end{array}\right|\label{eq:2sdet}
\end{gather}
The $k$-th step ($k=1,2,\dots,r$): 
\begin{gather}
\left|\mbf{R}_r\dots\mbf{R}_2\mbf{R}_1\mbf{Z}\mbf{R}_1^T\mbf{R}_2^T\dots\mbf{R}_k^T\right| =\notag\\
	\left|\arraycolsep=0.3pt\def\arraystretch{1.2}\begin{array}{cccccccc}
		\scriptstyle  f &\scriptstyle \dots &\scriptstyle  (D^2+1)^{k-1}f & \scriptstyle (D^2+1)^kf & \scriptstyle D^2(D^2+1)^kf& \scriptstyle \dots & \scriptstyle  D^{2(r-k)}(D^2+1)^kf\\
		\scriptstyle  (D^2+1)f &\scriptstyle \dots &\scriptstyle  (D^2+1)^kf &\scriptstyle  (D^2+1)^{k+1}f & \scriptstyle D^2(D^2+1)^{k+1}f &\scriptstyle \dots & \scriptstyle D^{2(r-k)}(D^2+1)^{k+1}f\\
		\scriptstyle  (D^2+1)^2f &\scriptstyle \dots &\scriptstyle  (D^2+1)^{k+1}f &\scriptstyle  (D^2+1)^{k+2}f & \scriptstyle D^2(D^2+1)^{k+2}f &\scriptstyle \dots & \scriptstyle D^{2(r-k)}(D^2+1)^{k+2}f\\		
		\scriptstyle \vdots &\scriptstyle \ddots  & \scriptstyle \vdots & \scriptstyle\vdots &\scriptstyle \vdots & \scriptstyle \ddots &\scriptstyle \vdots \\
		\scriptstyle  (D^2+1)^{m}f &\scriptstyle\dots &\scriptstyle  (D^2+1)^{m+k-1}f &\scriptstyle  (D^2+1)^{m+k}f&\scriptstyle  D^2(D^2+1)^{m+k}f &\scriptstyle\dots& \scriptstyle D^{2(r-k)}(D^2+1)^{m+k}f\\
		\scriptstyle \vdots &\scriptstyle \ddots  & \scriptstyle \vdots & \scriptstyle\vdots &\scriptstyle \vdots & \scriptstyle\ddots &\scriptstyle \vdots\\
		\scriptstyle  (D^2+1)^{r-1}f &\scriptstyle\dots &\scriptstyle  (D^2+1)^{k+r-2}f &\scriptstyle  (D^2+1)^{k+r-1}f&\scriptstyle  D^2(D^2+1)^{k+r-1}f&\scriptstyle\dots& \scriptstyle D^{2(r-k)}(D^2+1)^{k+r-1}f\\
\scriptstyle  (D^2+1)^{r}f &\scriptstyle \dots &\scriptstyle (D^2+1)^{k+r-1}f &\scriptstyle (D^2+1)^{k+r}f& \scriptstyle D^2(D^2+1)^{k+r}f&\scriptstyle \dots &\scriptstyle  D^{2(r-k)}(D^2+1)^{k+r}f
	\end{array}\right|\label{eq:ksdet}
\end{gather}
Finally, if $k=r$ we get the determinant (\ref{eq:ase}).\par
If we denote $\mbf{R}=\mbf{R}_r\mbf{R}_{r-1}\dots\mbf{R}_2\mbf{R}_1$ then
\begin{equation}\label{eq:reslem2}
|\mbf{Z}|=|\mbf{R}\mbf{Z}\mbf{R}^T|
\end{equation}
\end{proof}
\end{lemma}
\begin{remark}\label{rele1}
Let us put in (\ref{eq:ASD}) instead of $f$ the $D^\nu f$, $\nu=0,1,2,\dots$. We get for the determinant
\begin{gather}
	\left|\mbf{Z}^\ast\right|=\left| \begin{array}{ccccc}
		\scriptstyle  D^\nu f &\scriptstyle  D^2(D^\nu f) &\scriptstyle  D^4(D^\nu f) &\dots &\scriptstyle  D^{2r}(D^\nu f)\\
		\scriptstyle D^2(D^\nu f)&\scriptstyle D^4(D^\nu f)&\scriptstyle D^6(D^\nu f)  & \dots &\scriptstyle  D^{2r+2}(D^\nu f) \\
		\scriptstyle D^4(D^\nu f)&\scriptstyle D^6(D^\nu f)&\scriptstyle D^8(D^\nu f) & \dots &\scriptstyle D^{2r+4}(D^\nu f)\\
		\scriptstyle \vdots &\scriptstyle \vdots  & \scriptstyle \vdots & \scriptstyle\ddots & \scriptstyle \vdots\\
		\scriptstyle D^{2m}(D^\nu f) & \scriptstyle D^{2m+2}(D^\nu f)& \scriptstyle D^{2m+4}(D^\nu f)& \scriptstyle \dots & \scriptstyle D^{2(m+r)}(D^\nu f)\\
		\scriptstyle \vdots &\scriptstyle \vdots  & \scriptstyle \vdots & \scriptstyle\ddots & \scriptstyle \vdots\\
		\scriptstyle D^{2r-2}(D^\nu f) & \scriptstyle D^{2r}(D^\nu f)& \scriptstyle D^{2r+2}(D^\nu f)& \scriptstyle \dots &\scriptstyle D^{4r-2}(D^\nu f)\\
		\scriptstyle D^{2r}(D^\nu f) & \scriptstyle D^{2r+2}(D^\nu f)& \scriptstyle D^{2r+4}(D^\nu f)& \scriptstyle \dots &\scriptstyle D^{4r}(D^\nu f)
	\end{array}\label{eq:ASDast}\right|%\\
\end{gather}
\begin{gather}
=\left| \begin{array}{ccccc}
	\scriptstyle  D^\nu f &\scriptstyle  D^\nu D^2f &\scriptstyle D^\nu D^4f &\dots &\scriptstyle  D^\nu D^{2r}f\\
	\scriptstyle D^\nu D^2f&\scriptstyle D^\nu D^4f&\scriptstyle D^\nu D^6f  & \dots &\scriptstyle  D^\nu D^{2r+2}f \\
	\scriptstyle D^\nu D^4f&\scriptstyle D^\nu D^6f&\scriptstyle D^\nu D^8f & \dots &\scriptstyle D^\nu D^{2r+4}f\\
	\scriptstyle \vdots &\scriptstyle \vdots  & \scriptstyle \vdots & \scriptstyle\ddots & \scriptstyle \vdots\\
	\scriptstyle D^\nu D^{2m}f & \scriptstyle D^\nu D^{2m+2}f& \scriptstyle D^\nu D^{2m+4} f& \scriptstyle \dots & \scriptstyle D^\nu D^{2(m+r)}f\\
	\scriptstyle \vdots &\scriptstyle \vdots  & \scriptstyle \vdots & \scriptstyle\ddots & \scriptstyle \vdots\\
	\scriptstyle D^\nu D^{2r-2}f & \scriptstyle D^\nu D^{2r}f& \scriptstyle D^\nu D^{2r+2} f& \scriptstyle \dots &\scriptstyle D^\nu D^{4r-2}f\\
	\scriptstyle D^\nu D^{2r}f & \scriptstyle D^\nu D^{2r+2}f& \scriptstyle D^\nu D^{2r+4} f& \scriptstyle \dots &\scriptstyle D^\nu D^{4r}f
\end{array}\label{eq:ASDast1}\right|
\end{gather}
\begin{gather}
=\left|\arraycolsep=0.15pt\def\arraystretch{1.2}
	\begin{array}{ccccc}
		\scriptstyle  D^\nu f &\scriptstyle  (D^2+1)(D^\nu f) &\scriptstyle  
		(D^2+1)^2(D^\nu f) &\dots &\scriptstyle  (D^2+1)^r(D^\nu f)\\
		\scriptstyle  (D^2+1)(D^\nu f) &\scriptstyle  (D^2+1)^2(D^\nu f) &\scriptstyle  (D^2+1)^3(D^\nu f) &\dots&\scriptstyle  (D^2+1)^{r+1}(D^\nu f)\\
		\scriptstyle  (D^2+1)^2(D^\nu f) &\scriptstyle  (D^2+1)^3(D^\nu f) &\scriptstyle  (D^2+1)^4(D^\nu f) &\dots & \scriptstyle  (D^2+1)^{r+2}(D^\nu f)\\
		\scriptstyle \vdots &\scriptstyle \vdots  & \scriptstyle \vdots & \scriptstyle\ddots &\scriptstyle \vdots\\
		\scriptstyle  (D^2+1)^{m}(D^\nu f) &\scriptstyle  (D^2+1)^{m+1}(D^\nu f) &\scriptstyle  (D^2+1)^{m+2}(D^\nu f) &\dots&\scriptstyle  (D^2+1)^{m+r}(D^\nu f)\\
		\scriptstyle \vdots &\scriptstyle \vdots  & \scriptstyle \vdots & \scriptstyle\ddots & \scriptstyle \vdots\\
		\scriptstyle  (D^2+1)^{r-1}(D^\nu f) &\scriptstyle  (D^2+1)^{r}(D^\nu f) &\scriptstyle  (D^2+1)^{r+1}(D^\nu f) &\dots&\scriptstyle  (D^2+1)^{2r-1}(D^\nu f)\\
		\scriptstyle  (D^2+1)^{r}(D^\nu f) &\scriptstyle  (D^2+1)^{r+1}(D^\nu f) &\scriptstyle  (D^2+1)^{r+2}(D^\nu f) &\dots &\scriptstyle  (D^2+1)^{2r}(D^\nu f)
	\end{array}\right|\label{eq:aseast}%\\
\end{gather}
\begin{gather}
=\left|\arraycolsep=0.15pt\def\arraystretch{1.2}
\begin{array}{ccccc}
	\scriptstyle  D^\nu f &\scriptstyle  D^\nu(D^2+1)f &\scriptstyle  D^\nu(D^2+1)^2f &\dots &\scriptstyle  D^\nu(D^2+1)^rf\\
	\scriptstyle  D^\nu(D^2+1)f &\scriptstyle  D^\nu(D^2+1)^2f &\scriptstyle  D^\nu(D^2+1)^3 f &\dots&\scriptstyle  D^\nu(D^2+1)^{r+1}f\\
	\scriptstyle  D^\nu(D^2+1)^2f &\scriptstyle  D^\nu(D^2+1)^3f &\scriptstyle  D^\nu(D^2+1)^4f &\dots & \scriptstyle  D^\nu(D^2+1)^{r+2}f\\
	\scriptstyle \vdots &\scriptstyle \vdots  & \scriptstyle \vdots & \scriptstyle\ddots &\scriptstyle \vdots\\
	\scriptstyle  D^\nu(D^2+1)^{m}f &\scriptstyle  D^\nu(D^2+1)^{m+1}f &\scriptstyle  D^\nu(D^2+1)^{m+2}f &\dots&\scriptstyle  D^\nu(D^2+1)^{m+r}f\\
	\scriptstyle \vdots &\scriptstyle \vdots  & \scriptstyle \vdots & \scriptstyle\ddots & \scriptstyle \vdots\\
	\scriptstyle  D^\nu(D^2+1)^{r-1}f &\scriptstyle  D^\nu(D^2+1)^{r}f &\scriptstyle  D^\nu(D^2+1)^{r+1}f &\dots&\scriptstyle  D^\nu(D^2+1)^{2r-1}f\\
	\scriptstyle  D^\nu(D^2+1)^{r}f &\scriptstyle  D^\nu(D^2+1)^{r+1}f &\scriptstyle  D^\nu(D^2+1)^{r+2}f &\dots &\scriptstyle  D^\nu(D^2+1)^{2r}f
\end{array}\right|\label{eq:aseast1}
\end{gather}
\end{remark}

\begin{theorem}\label{wr}
Let $D$ be a differential operator and $f$ be a $4n-2$ ($n\in \mathbb{N}$) times differentiable function on an interval $I$, then the Wronskian 
$$W(f,Df,D^2f\dots,D^{2n-2}f,D^{2n-1}f)$$ 
has the relationship 
\begin{gather}
\scriptstyle W(f,Df,\dots,D^{2n-2}f,D^{2n-1}f)\notag \\[0.5ex]
=\left|\begin{array}{cccccc}
	\scriptstyle  f &\scriptstyle  Df &\scriptstyle  D^2f &\scriptstyle\dots& \scriptstyle D^{2n-2}f &\scriptstyle  D^{2n-1}f\\
	\scriptstyle Df &\scriptstyle D^2f &\scriptstyle D^3f & \scriptstyle \dots &\scriptstyle D^{2n-1}f  &\scriptstyle  D^{2n}f \\
	\scriptstyle D^2f &\scriptstyle D^3f &\scriptstyle D^4f & \scriptstyle\dots &\scriptstyle D^{2n}f &\scriptstyle D^{2n+1}f\\
	\scriptstyle \vdots &\scriptstyle \vdots  & \scriptstyle \vdots & \scriptstyle\ddots &\scriptstyle \vdots & \scriptstyle \vdots\\
	\scriptstyle D^{2n-2}f & \scriptstyle D^{2n-1}f& \scriptstyle D^{2n}f & \scriptstyle \dots & \scriptstyle D^{4n-4}f &\scriptstyle D^{4n-3}f\\
	\scriptstyle D^{2n-1}f & \scriptstyle D^{2n}f& \scriptstyle D^{2n+1}f & \scriptstyle \dots & \scriptstyle D^{4n-3}f &\scriptstyle D^{4n-2}f
\end{array} \right|\label{eq:det1}\\
=\left|\arraycolsep=0.3pt\def\arraystretch{1.2}
\begin{array}{ccccccc}
	\scriptstyle  f &\scriptstyle  Df &\scriptstyle  (D^2+1)f &\scriptstyle  D(D^2+1)f &\scriptstyle\dots& \scriptstyle (D^2+1)^{n-1}f &\scriptstyle  D(D^2+1)^{n-1}f\\
	\scriptstyle Df&\scriptstyle D^2f&\scriptstyle D(D^2+1)f & \scriptstyle D^2(D^2+1)f & \scriptstyle \dots &\scriptstyle D(D^2+1)^{n-1}f  &\scriptstyle  D^2(D^2+1)^{n-1}f \\
	\scriptstyle (D^2+1)f&\scriptstyle D(D^2+1)f&\scriptstyle (D^2+1)^2f&\scriptstyle D(D^2+1)^2f& \scriptstyle\dots &\scriptstyle (D^2+1)^nf&\scriptstyle D(D^2+1)^nf\\
	\scriptstyle D(D^2+1)f&\scriptstyle D^2(D^2+1)f&\scriptstyle D(D^2+1)^2f &\scriptstyle  D^2(D^2+1)^2f& \scriptstyle\dots & \scriptstyle D(D^2+1)^nf  & \scriptstyle D^2(D^2+1)^nf \\
	\scriptstyle \vdots &\scriptstyle \vdots  & \scriptstyle \vdots &\scriptstyle \vdots  & \scriptstyle\ddots &\scriptstyle \vdots & \scriptstyle \vdots\\
	\scriptstyle (D^2+1)^{n-1}f & \scriptstyle D(D^2+1)^{n-1}f& \scriptstyle (D^2+1)^nf& \scriptstyle D(D^2+1)^nf& \scriptstyle \dots & \scriptstyle (D^2+1)^{2n-2}f&\scriptstyle D(D^2+1)^{2n-2}f\\
	\scriptstyle D(D^2+1)^{n-1}f & \scriptstyle D^2(D^2+1)^{n-1}f& \scriptstyle D(D^2+1)^nf& \scriptstyle D^2(D^2+1)^nf& \scriptstyle \dots & \scriptstyle D(D^2+1)^{2n-2}f&\scriptstyle D^2(D^2+1)^{2n-2}f
\end{array} \right|\label{eq:det2}\\
=\left|w_{ij}\right|_{(2n)\times (2n)}=\mbf{W}\notag
\end{gather}
where 
\begin{displaymath}
	w_{ij}= \left\{\begin{array}{cl}
		(D^2+1)^kf & \textrm{if $i,j$ are both odd and $k=\frac{i+j}2-1$}\\[0.5ex]
		D^2(D^2+1)^kf & \textrm{if $i,j$ are both even and $k=\frac{i+j}{2}-2$}\\[0.5ex]
		D(D^2+1)^kf & \textrm{if $i+j$ is odd and $k=\frac{i+j-3}{2}$}
	\end{array} \right.
\end{displaymath}
\end{theorem}
\begin{proof} We will describe and prove a process involving using the row and column operations of (\ref{eq:det1}) that transforms the determinant (\ref{eq:det1}) into (\ref{eq:det2}). First, we will perform $n$ row procedures on the matrix of determinant, then the same number of column procedures after the resulting row procedures. 
We denote the operation of adding the p-th row (resp. column) to the q-th row (resp. column) by $R_p+R_q \to R_q$ (resp. $C_p+C_q \to C_q$).\par
The first procedure: We start with the matrix of determinant (\ref{eq:det1}).
$$R_{k-2}+R_k\to R_k, \quad k=2n,2n-1,\dots,4,3$$
The second procedure, on the matrix of determinant that we got after the first procedure.
$$R_{k-2}+R_k\to R_k, \quad k=2n,2n-1,\dots,6,5$$
The $r$-th procedure ($r=1,2,\dots,n-1$)
$$R_{k-2}+R_k\to R_k, \quad k=2n,2n-1,\dots,2r, 2r+1$$
The $n$-th procedure
$$R_{2n-2}+R_{2n}\to R_{2n}$$
We observe that we only add even rows to even and odd to odd rows.\\
With the resulting matrix of determinant, we perform analogous procedures with the columns.\\
The $r$-th procedure ($r=1,2,\dots,n-1$)
$$C_{k-2}+C_k\to C_k, \quad k=2n,2n-1,\dots,2r, 2r+1$$
The $n$-th procedure
$$C_{2n-2}+C_{2n}\to C_{2n}$$
In this case too, we observe that even columns are added only to even and odd columns to odd columns. It follows that the determinant of the matrix (\ref{eq:det1}) after the described row and column procedures can be calculated by first dividing the matrix of the determinant into four disjoint subdeterminants by choosing:
\begin{itemize}
	\item[-] odd rows and odd columns (next determinant OROC)
	\item[-] odd rows and even columns (next determinant OREC)
	\item[-] even rows and odd columns (next determinant EROC)
	\item[-] even rows and even columns (next determinant EREC)
\end{itemize}
and apply Lemma \ref{lmm1} to them. We then combine them into the resulting matrix of the determinant in the opposite way than they were divided.\\
We select entries that are in the odd rows and odd columns of (\ref{eq:det1}) and apply Lemma \ref{lmm1}:
\begin{gather}
OROC=\left|
\begin{array}{cccccc}
	\scriptstyle  f &\scriptstyle  D^2f &\scriptstyle  D^4f &\scriptstyle\dots& \scriptstyle D^{2n-4}f &\scriptstyle  D^{2n-2}f\\
	\scriptstyle D^2f &\scriptstyle D^4f &\scriptstyle D^6f & \scriptstyle \dots &\scriptstyle D^{2n-2}f  &\scriptstyle  D^{2n}f \\
	\scriptstyle D^4f &\scriptstyle D^6f &\scriptstyle D^8f & \scriptstyle\dots &\scriptstyle D^{2n}f &\scriptstyle D^{2n+2}f\\
	\scriptstyle \vdots &\scriptstyle \vdots  & \scriptstyle \vdots & \scriptstyle\ddots &\scriptstyle \vdots & \scriptstyle \vdots\\
	\scriptstyle D^{2n-4}f & \scriptstyle D^{2n-2}f& \scriptstyle D^{2n}f & \scriptstyle \dots & \scriptstyle D^{4n-8}f &\scriptstyle D^{4n-6}f\\
	\scriptstyle D^{2n-2}f & \scriptstyle D^{2n}f& \scriptstyle D^{2n+2}f & \scriptstyle \dots & \scriptstyle D^{4n-6}f &\scriptstyle D^{4n-4}f
\end{array} \right|\label{eq:OROC1}\\
=\left|\arraycolsep=1.5pt\def\arraystretch{1.2}
\begin{array}{cccccc}
	\scriptstyle  f &\scriptstyle  (D^2+1)f &\scriptstyle  (D^2+1)^2f &\scriptstyle\dots& \scriptstyle (D^2+1)^{n-2}f &\scriptstyle  (D^2+1)^{n-1}f\\
	\scriptstyle (D^2+1)f &\scriptstyle (D^2+1)^2f &\scriptstyle (D^2+1)^3f & \scriptstyle \dots &\scriptstyle (D^2+1)^{n-1}f  &\scriptstyle  (D^2+1)^nf \\
	\scriptstyle (D^2+1)^2f &\scriptstyle (D^2+1)^3f &\scriptstyle (D^2+1)^4f & \scriptstyle\dots &\scriptstyle (D^2+1)^nf &\scriptstyle (D^2+1)^{n+1}f\\
	\scriptstyle \vdots &\scriptstyle \vdots  & \scriptstyle \vdots & \scriptstyle\ddots &\scriptstyle \vdots & \scriptstyle \vdots\\
	\scriptstyle (D^2+1)^{n-2}f & \scriptstyle (D^2+1)^{n-1}f& \scriptstyle (D^2+1)^nf & \scriptstyle \dots & \scriptstyle (D^2+1)^{2n-4}f &\scriptstyle (D^2+1)^{2n-3}f\\
	\scriptstyle (D^2+1)^{n-1}f & \scriptstyle (D^2+1)^nf& \scriptstyle (D^2+1)^{n+1}f & \scriptstyle \dots & \scriptstyle (D^2+1)^{2n-3}f &\scriptstyle (D^2+1)^{2n-2}f
\end{array} \right|\label{eq:OROC2}
\end{gather}
We select entries that are in the odd rows and even columns of (\ref{eq:det1}) and apply the Remark \ref{rele1} (see (\ref{eq:ASDast1}) and (\ref{eq:aseast1})) for $\nu=1$. 
 So we get
\begin{gather}
OREC=\left|
\begin{array}{cccccc}
	\scriptstyle  Df &\scriptstyle  D^3f &\scriptstyle  D^5f &\scriptstyle\dots& \scriptstyle D^{2n-3}f &\scriptstyle  D^{2n-1}f\\
	\scriptstyle D^3f &\scriptstyle D^5f &\scriptstyle D^7f & \scriptstyle \dots &\scriptstyle D^{2n-1}f  &\scriptstyle  D^{2n+1}f \\
	\scriptstyle D^5f &\scriptstyle D^7f &\scriptstyle D^9f & \scriptstyle\dots &\scriptstyle D^{2n+1}f &\scriptstyle D^{2n+3}f\\
	\scriptstyle \vdots &\scriptstyle \vdots  & \scriptstyle \vdots & \scriptstyle\ddots &\scriptstyle \vdots & \scriptstyle \vdots\\
	\scriptstyle D^{2n-3}f & \scriptstyle D^{2n-1}f& \scriptstyle D^{2n+1}f & \scriptstyle \dots & \scriptstyle D^{4n-7}f &\scriptstyle D^{4n-5}f\\
	\scriptstyle D^{2n-1}f & \scriptstyle D^{2n+1}f& \scriptstyle D^{2n+3}f & \scriptstyle \dots & \scriptstyle D^{4n-5}f &\scriptstyle D^{4n-3}f
\end{array} \right|\label{eq:OREC1}
\end{gather}
\begin{gather}
=\left|
\begin{array}{cccccc}
	\scriptstyle  Df &\scriptstyle  DD^2f &\scriptstyle  DD^4f &\scriptstyle\dots& \scriptstyle DD^{2n-4}f &\scriptstyle  DD^{2n-2}f\\
	\scriptstyle DD^2f &\scriptstyle DD^4f &\scriptstyle DD^6f & \scriptstyle \dots &\scriptstyle DD^{2n-2}f  &\scriptstyle  DD^{2n}f \\
	\scriptstyle DD^4f &\scriptstyle DD^6f &\scriptstyle DD^8f & \scriptstyle\dots &\scriptstyle DD^{2n}f &\scriptstyle DD^{2n+2}f\\
	\scriptstyle \vdots &\scriptstyle \vdots  & \scriptstyle \vdots & \scriptstyle\ddots &\scriptstyle \vdots & \scriptstyle \vdots\\
	\scriptstyle DD^{2n-4}f & \scriptstyle DD^{2n-2}f& \scriptstyle DD^{2n}f & \scriptstyle \dots & \scriptstyle DD^{4n-8}f &\scriptstyle DD^{4n-6}f\\
	\scriptstyle DD^{2n-2}f & \scriptstyle DD^{2n}f& \scriptstyle DD^{2n+2}f & \scriptstyle \dots & \scriptstyle DD^{4n-6}f &\scriptstyle DD^{4n-4}f
\end{array} \right|\label{eq:OREC11}
\end{gather}
\begin{gather}
=\left|\arraycolsep=1pt\def\arraystretch{1.2}
\begin{array}{cccccc}
	\scriptstyle  Df &\scriptstyle  D(D^2+1)f &\scriptstyle  D(D^2+1)^2f &\scriptstyle\dots& \scriptstyle D(D^2+1)^{n-2}f &\scriptstyle  D(D^2+1)^{n-1}f\\
	\scriptstyle D(D^2+1)f &\scriptstyle D(D^2+1)^2f &\scriptstyle D(D^2+1)^3f & \scriptstyle \dots &\scriptstyle D(D^2+1)^{n-1}f  &\scriptstyle  D(D^2+1)^nf \\
	\scriptstyle D(D^2+1)^2f &\scriptstyle D(D^2+1)^3f &\scriptstyle D(D^2+1)^4f & \scriptstyle\dots &\scriptstyle D(D^2+1)^nf &\scriptstyle D(D^2+1)^{n+1}f\\
	\scriptstyle \vdots &\scriptstyle \vdots  & \scriptstyle \vdots & \scriptstyle\ddots &\scriptstyle \vdots & \scriptstyle \vdots\\
	\scriptstyle D(D^2+1)^{n-2}f & \scriptstyle D(D^2+1)^{n-1}f& \scriptstyle D(D^2+1)^nf & \scriptstyle \dots & \scriptstyle D(D^2+1)^{2n-4}f &\scriptstyle D(D^2+1)^{2n-3}f\\
	\scriptstyle D(D^2+1)^{n-1}f & \scriptstyle D(D^2+1)^nf& \scriptstyle D(D^2+1)^{n+1}f & \scriptstyle \dots & \scriptstyle D(D^2+1)^{2n-3}f &\scriptstyle D(D^2+1)^{2n-2}f
\end{array} \right|\label{eq:OREC2}
\end{gather}
We select entries that are in the even rows and odd columns of (\ref{eq:det1}) and apply Lemma \ref{lmm1}.
\begin{equation}\label{eq:EROC}
EROC=OROC
\end{equation}
We select entries that are in the even rows and even columns of (\ref{eq:det1}) and apply the Remark \ref{rele1} (see (\ref{eq:ASDast1}) and (\ref{eq:aseast1})) for $\nu=2$. We have
\begin{gather}
	EREC=\left|
	\begin{array}{cccccc}
		\scriptstyle  D^2f &\scriptstyle  D^4f &\scriptstyle  D^6f &\scriptstyle\dots& \scriptstyle D^{2n-2}f &\scriptstyle  D^{2n}f\\
		\scriptstyle D^4f &\scriptstyle D^6f &\scriptstyle D^8f & \scriptstyle \dots &\scriptstyle D^{2n}f  &\scriptstyle  D^{2n+2}f \\
		\scriptstyle D^6f &\scriptstyle D^8f &\scriptstyle D^10f & \scriptstyle\dots &\scriptstyle D^{2n+2}f &\scriptstyle D^{2n+4}f\\
		\scriptstyle \vdots &\scriptstyle \vdots  & \scriptstyle \vdots & \scriptstyle\ddots &\scriptstyle \vdots & \scriptstyle \vdots\\
		\scriptstyle D^{2n-2}f & \scriptstyle D^{2n}f& \scriptstyle D^{2n+2}f & \scriptstyle \dots & \scriptstyle D^{4n-6}f &\scriptstyle D^{4n-4}f\\
		\scriptstyle D^{2n}f & \scriptstyle D^{2n+2}f& \scriptstyle D^{2n+4}f & \scriptstyle \dots & \scriptstyle D^{4n-4}f &\scriptstyle D^{4n-2}f
	\end{array} \right|\label{eq:EREC1}
\end{gather}
\begin{gather}=\left|
\begin{array}{cccccc}
	\scriptstyle  D^2f &\scriptstyle  D^2D^2f &\scriptstyle  D^2D^4f &\scriptstyle\dots& \scriptstyle D^2D^{2n-4}f &\scriptstyle  D^2D^{2n-2}f\\
	\scriptstyle D^2D^2f &\scriptstyle D^2D^4f &\scriptstyle D^2D^6f & \scriptstyle \dots &\scriptstyle D^2D^{2n-2}f  &\scriptstyle  D^2D^{2n}f \\
	\scriptstyle D^2D^4f &\scriptstyle D^2D^6f &\scriptstyle D^2D^8f & \scriptstyle\dots &\scriptstyle D^2D^{2n}f &\scriptstyle D^2D^{2n+2}f\\
	\scriptstyle \vdots &\scriptstyle \vdots  & \scriptstyle \vdots & \scriptstyle\ddots &\scriptstyle \vdots & \scriptstyle \vdots\\
	\scriptstyle D^2D^{2n-4}f & \scriptstyle D^2D^{2n-2}f& \scriptstyle D^2D^{2n}f & \scriptstyle \dots & \scriptstyle D^2D^{4n-8}f &\scriptstyle D^2D^{4n-6}f\\
	\scriptstyle D^2D^{2n-2}f & \scriptstyle D^2D^{2n}f& \scriptstyle D^2D^{2n+2}f & \scriptstyle \dots & \scriptstyle D^2D^{4n-6}f &\scriptstyle D^2D^{4n-4}f
\end{array} \right|\label{eq:EREC11}\\
=\left|\arraycolsep=0.5pt\def\arraystretch{1.2}
\begin{array}{cccccc}
		\scriptstyle  D^2f &\scriptstyle  D^2(D^2+1)f &\scriptstyle  D^2(D^2+1)^2f &\scriptstyle\dots& \scriptstyle D^2(D^2+1)^{n-2}f &\scriptstyle  D^2(D^2+1)^{n-1}f\\
		\scriptstyle D^2(D^2+1)f &\scriptstyle D^2(D^2+1)^2f &\scriptstyle D^2(D^2+1)^3f & \scriptstyle \dots &\scriptstyle D^2(D^2+1)^{n-1}f  &\scriptstyle  D^2(D^2+1)^nf \\
		\scriptstyle D^2(D^2+1)^2f &\scriptstyle D^2(D^2+1)^3f &\scriptstyle D^2(D^2+1)^4f & \scriptstyle\dots &\scriptstyle D^2(D^2+1)^nf &\scriptstyle D^2(D^2+1)^{n+1}f\\
		\scriptstyle \vdots &\scriptstyle \vdots  & \scriptstyle \vdots & \scriptstyle\ddots &\scriptstyle \vdots & \scriptstyle \vdots\\
		\scriptstyle D^2(D^2+1)^{n-2}f & \scriptstyle D^2(D^2+1)^{n-1}f& \scriptstyle D^2(D^2+1)^nf & \scriptstyle \dots & \scriptstyle D^2(D^2+1)^{2n-4}f &\scriptstyle D^2(D^2+1)^{2n-3}f\\
		\scriptstyle D^2(D^2+1)^{n-1}f & \scriptstyle D^2(D^2+1)^nf& \scriptstyle D^2(D^2+1)^{n+1}f & \scriptstyle \dots & \scriptstyle D^2(D^2+1)^{2n-3}f &\scriptstyle D^2(D^2+1)^{2n-2}f
	\end{array} \right|\label{eq:EREC2}
\end{gather}
By inserting the entries of matrices of determinants (\ref{eq:OROC2}), (\ref{eq:OREC2}) (\ref{eq:EROC}) end (\ref{eq:EREC2}) into the $2n\times 2n$ matrix of the determinant in a reverse way, we get the determinant (\ref{eq:det2}).
\end{proof}

\textbf{Another proof of Theorem \ref{wr}}\par
\vspace{3mm}
Let us denote
\begin{equation}\label{Ukn}
\mbf{U}_{k,n}=[u_{k_{ij}}]_{n\times n}, \  u_{k_{i,j}}=\left\{\begin{array}{cl} 1 & \textrm{if $i=j$ or ($i=j+2$ and $2k+1\le i\le n$}) \\[0.5ex]
	0 & \textrm{otherwise} \end{array} \right.
\end{equation}
$k=1,2,\dots, \left[\frac{n+1}2\right]-1$, $n=3,4,\dots$

and let
\[\mbf{V}=
\left[\begin{array}{cccccc}
	\scriptstyle  f &\scriptstyle  Df &\scriptstyle  D^2f &\scriptstyle\dots& \scriptstyle D^{2n-2}f &\scriptstyle  D^{2n-1}f\\
	\scriptstyle Df &\scriptstyle D^2f &\scriptstyle D^3f & \scriptstyle \dots &\scriptstyle D^{2n-1}f  &\scriptstyle  D^{2n}f \\
	\scriptstyle D^2f &\scriptstyle D^3f &\scriptstyle D^4f & \scriptstyle\dots &\scriptstyle D^{2n}f &\scriptstyle D^{2n+1}f\\
	\scriptstyle \vdots &\scriptstyle \vdots  & \scriptstyle \vdots & \scriptstyle\ddots &\scriptstyle \vdots & \scriptstyle \vdots\\
	\scriptstyle D^{2n-2}f & \scriptstyle D^{2n-1}f& \scriptstyle D^{2n}f & \scriptstyle \dots & \scriptstyle D^{4n-4}f &\scriptstyle D^{4n-3}f\\
	\scriptstyle D^{2n-1}f & \scriptstyle D^{2n}f& \scriptstyle D^{2n+1}f & \scriptstyle \dots & \scriptstyle D^{4n-3}f &\scriptstyle D^{4n-2}f
\end{array} \right]\]
then
\begin{gather}
W(f,Df,D^2f,\dots,D^{2n-2}f,D^{2n-1}f)=|UVU^T|=\notag\\
	\left[\arraycolsep=0.3pt\def\arraystretch{1.2}
	\begin{array}{ccccccc}
		\scriptstyle  f &\scriptstyle  Df &\scriptstyle  (D^2+1)f &\scriptstyle  D(D^2+1)f &\scriptstyle\dots& \scriptstyle (D^2+1)^{n-1}f &\scriptstyle  D(D^2+1)^{n-1}f\\
		\scriptstyle Df&\scriptstyle D^2f&\scriptstyle D(D^2+1)f & \scriptstyle D^2(D^2+1)f & \scriptstyle \dots &\scriptstyle D(D^2+1)^{n-1}f  &\scriptstyle  D^2(D^2+1)^{n-1}f \\
		\scriptstyle (D^2+1)f&\scriptstyle D(D^2+1)f&\scriptstyle (D^2+1)^2f&\scriptstyle D(D^2+1)^2f& \scriptstyle\dots &\scriptstyle (D^2+1)^nf&\scriptstyle D(D^2+1)^nf\\
		\scriptstyle D(D^2+1)f&\scriptstyle D^2(D^2+1)f&\scriptstyle D(D^2+1)^2f &\scriptstyle  D^2(D^2+1)^2f& \scriptstyle\dots & \scriptstyle D(D^2+1)^nf  & \scriptstyle D^2(D^2+1)^nf \\
		\scriptstyle \vdots &\scriptstyle \vdots  & \scriptstyle \vdots &\scriptstyle \vdots  & \scriptstyle\ddots &\scriptstyle \vdots & \scriptstyle \vdots\\
		\scriptstyle (D^2+1)^{n-1}f & \scriptstyle D(D^2+1)^{n-1}f& \scriptstyle (D^2+1)^nf& \scriptstyle D(D^2+1)^nf& \scriptstyle \dots & \scriptstyle (D^2+1)^{2n-2}f&\scriptstyle D(D^2+1)^{2n-2}f\\
		\scriptstyle D(D^2+1)^{n-1}f & \scriptstyle D^2(D^2+1)^{n-1}f& \scriptstyle D(D^2+1)^nf& \scriptstyle D^2(D^2+1)^nf& \scriptstyle \dots & \scriptstyle D(D^2+1)^{2n-2}f&\scriptstyle D^2(D^2+1)^{2n-2}f
	\end{array} \right]\label{eq:UknV}
\end{gather}
where
$\mbf{U}=\mbf{U}_{n-1,2n}\mbf{U}_{n-2,2n}\dots \mbf{U}_{2,2n}\mbf{U}_{1,2n}$.

The formula (\ref{eq:UknV}) includes all the elementary rows and columns
 operations that were used in the proof of Theorem \ref{wr}.
\begin{theorem}\label{wrf}
Let $n$ be any fixed natural number and 
\begin{itemize}
\item[(i)] $f(x)=x^n\sin x$, then the functions
$$ f(x), Df(x), D^2f(x),\dots,D^{2n}f(x), D^{2n+1}f(x)$$
\item[(ii)] $g(x)=x^n\cos x$, then the functions
$$ g(x), Dg(x), D^2g(x),\dots,D^{2n}g(x), D^{2n+1}g(x)$$
\end{itemize}
are linearly independent. 
\end{theorem}
\begin{proof}
The proof for cases (i) and (ii) is the same. For the sake of clarity of notation, let us solve case (i).
The Wronskian 
$$W(f,Df,D^2f\dots, D^{2n}f, D^{2n+1}f)$$
is a determinant of a $(2n+2)\times (2n+2)$ matrix.\par
It follows from Theorem \ref{wr} that
\begin{gather}
\begin{array}{ll}
\scriptstyle W(f,Df,D^2f\dots, D^{2n}f, D^{2n+1}f)\\[0.5ex]
&\\
=\left|\begin{array}{cccccc}
\scriptstyle  f &\scriptstyle  Df &\scriptstyle  (D^2+1)f &\scriptstyle \dots & \scriptstyle (D^2+1)^nf &\scriptstyle  D(D^2+1)^nf\\
\scriptstyle Df & \scriptstyle D^2f& \scriptstyle D(D^2+1)f & \scriptstyle \dots & \scriptstyle D(D^2+1)^nf  &\scriptstyle  D^2(D^2+1)^nf \\
\scriptstyle (D^2+1)f & \scriptstyle D(D^2+1)f & \scriptstyle (D^2+1)^2f & \scriptstyle \dots&\scriptstyle (D^2+1)^{n+1}f&\scriptstyle D(D^2+1)^{n+1}f\\
\scriptstyle D(D^2+1)f & \scriptstyle D^2(D^2+1)f & \scriptstyle D(D^2+1)^2f & \scriptstyle \dots&\scriptstyle D(D^2+1)^{n+1}f&\scriptstyle D^2(D^2+1)^{n+1}f\\
\scriptstyle \vdots &\scriptstyle \vdots&\scriptstyle \vdots &\scriptstyle \ddots & \scriptstyle \vdots  & \scriptstyle \vdots \\
\scriptstyle (D^2+1)^nf &\scriptstyle  D(D^2+1)^nf & \scriptstyle (D^2+1)^{n+1}f &\scriptstyle \dots  & \scriptstyle (D^2+1)^{2n}f & \scriptstyle D(D^2+1)^{2n}f\\
\scriptstyle D(D^2+1)^nf & \scriptstyle D^2(D^2+1)^nf& \scriptstyle D(D^2+1)^{n+1}f& \scriptstyle \dots & \scriptstyle D(D^2+1)^{2n}f&\scriptstyle D^2(D^2+1)^{2n}f
\end{array} \right|\\
\end{array}\label{maindet}
\end{gather}
 We have \hbox{$(D^2+1)^{n+1}f=0$} because $f(x)=x^n\sin x$ is a particular solution of the homogeneous differential equation $(D^2+1)^{n+1}f=0$. Also $(D^2+1)^kf= 0$ for $k>n+1$, but $(D^2+1)^kf\ne 0$ for $0\le k\le n$ because $f(x)=x^n\sin x$ is not the solution of the homogeneous differential equation $(D^2+1)^kf= 0$.\par
We will prove that
\begin{equation}\label{eq:mdet}
	\left|\begin{array}{cc}
		(D^2+1)^nf & D(D^2+1)^nf\\
		D(D^2+1)^nf  & D^2(D^2+1)^nf 
	\end{array} \right|\ne 0	
\end{equation}
Because $(D^2+1)^nf \ne 0$ and  $(D^2+1)^nf$ is not a linear function, we have
 $D(D^2+1)^nf\ne 0$ and  $D^2(D^2+1)^nf\ne 0$. 
Suppose the converse of the statement (\ref{eq:mdet}) is true, i.e.
$$(D^2+1)^nf\cdot D^2(D^2+1)^nf- (D(D^2+1)^nf)^2=0$$
We put $y=(D^2+1)^nf$, then we get the differential equation
\begin{equation}\label{DRT2} 
	yD^2y-(Dy)^2=0
\end{equation}
whose general solution is
\begin{equation}\label{eq:drov}
y=C_1e^{C_2x}
\end{equation}
where $C_1\ne 0$ since $(D^2+1)^nf=y\ne 0$. The solution (\ref{eq:drov}) is 
not in the vector space
\begin{equation}\label{span}
V=\text{span}\left(x^n\sin x, x^n\cos x,x^{n-1}\sin x,x^{n-1}\cos x,\dots,\sin x,\cos x\right)
\end{equation}
for any real numbers $C_1\ne 0$, $C_2$. So the equation (\ref{eq:mdet}) is true  and the determinant (\ref{maindet}) will have the form (\ref{uprmaindet})
\begin{landscape}
%\begin{displaymath}
\begin{equation}\label{uprmaindet}
	\begin{array}{rl}
		&\scriptstyle W(f,Df,D^2,\dots, D^{2n}f,D^{2n+1}f)=\\[0.5ex]
		&\left|\begin{array}{ccccccccc}
			\scriptstyle{f} & \scriptstyle{Df} & \scriptstyle\dots 
			&\scriptstyle (D^2+1)^{n-2}f & \scriptstyle D(D^2+1)^{n-2}f
			&\scriptstyle (D^2+1)^{n-1}f & \scriptstyle D(D^2+1)^{n-1}f & \scriptstyle{(D^2+1)^nf} & \scriptstyle{D(D^2+1)^nf}\\
			
			\scriptstyle{Df}&\scriptstyle{D^2f} &\scriptstyle \dots
			&\scriptstyle{D(D^2+1)^{n-2}f} &\scriptstyle D^2(D^2+1)^{n-2}f 
			&\scriptstyle{D(D^2+1)^{n-1}f} &\scriptstyle D^2(D^2+1)^{n-1}f 
			& \scriptstyle{D(D^2+1)^nf}  & \scriptstyle{D^2(D^2+1)^nf} \\
			
			\scriptstyle{(D^2+1)f}&\scriptstyle{D(D^2+1)f}&\scriptstyle \dots
			& \scriptstyle (D^2+1)^{n-1}f&\scriptstyle D(D^2+1)^{n-1}f
			& \scriptstyle (D^2+1)^{n}f&\scriptstyle D(D^2+1)^{n}f
			&\scriptstyle{0} &\scriptstyle{0}\\
						
			\scriptstyle{D(D^2+1)f}&\scriptstyle{D^2(D^2+1)f}&\scriptstyle\dots
			& \scriptstyle{D(D^2+1)^{n-1}f}&\scriptstyle{D^2(D^2+1)^{n-1}f} 
			& \scriptstyle{D(D^2+1)^nf}&\scriptstyle{D^2(D^2+1)^nf} 
			&\scriptstyle{0}  &\scriptstyle{0}\\
			
			\scriptstyle{(D^2+1)^2f}&\scriptstyle{D(D^2+1)^2f}&\scriptstyle\dots
			& \scriptstyle{(D^2+1)^{n}f}&\scriptstyle{D(D^2+1)^{n}f} 
			& \scriptstyle{0} &\scriptstyle{0}
			&\scriptstyle{0}  &\scriptstyle{0}\\
			
			\scriptstyle{D(D^2+1)^2f}&\scriptstyle{D^2(D^2+1)^2f}&\scriptstyle\dots
			& \scriptstyle{D(D^2+1)^{n}f}&\scriptstyle{D^2(D^2+1)^{n}f} 
			& \scriptstyle{0} &\scriptstyle{0}
			&\scriptstyle{0}  &\scriptstyle{0}\\
			
			\scriptstyle \vdots & 	\scriptstyle \vdots & 	\scriptstyle \ddots & 
			\scriptstyle \vdots & 	\scriptstyle \vdots & 	\scriptstyle \vdots & 
			\scriptstyle \vdots & 	\scriptstyle \vdots & 	\scriptstyle \vdots \\
			
			\scriptstyle{(D^2+1)^nf} & \scriptstyle{D(D^2+1)^nf} & \scriptstyle{0} &\scriptstyle{0} & \scriptstyle{0} & \scriptstyle{0}& \scriptstyle{0}
			&\scriptstyle 0 & \scriptstyle 0\\
			
			\scriptstyle{D(D^2+1)^nf} & \scriptstyle{D^2(D^2+1)^nf} & \scriptstyle{0} &\scriptstyle{0} & \scriptstyle{0} & \scriptstyle{0}& \scriptstyle{0}
			&\scriptstyle 0 & \scriptstyle 0\\		
		\end{array} \right|
	\end{array}
%\end{displaymath}
\end{equation}
\end{landscape}
\noindent which is not equal to zero with respect to (\ref{eq:mdet}).
The determinant can be evaluated by $n$-fold use of the Laplace generalized expansion. So 
\begin{displaymath}
W(f,Df,D^2f\dots, D^{2n}f, D^{2n+1}f)=
	\left|\begin{array}{cc}
		(D^2+1)^nf & D(D^2+1)^nf\\
		D(D^2+1)^nf  & D^2(D^2+1)^{n}f 
	\end{array} \right|^{n+1}\ne 0	
\end{displaymath}
\par
In the case (ii) will be  $(D^2+1)^kg=~0$ for $k\ge n+1$, $(D^2+1)^ng\ne 0$ for 
$0\le k\le n$ and
\begin{displaymath}
	W(g,Dg,D^2g\dots, D^{2n}g, D^{2n+1}g)=
	\left|\begin{array}{cc}
		(D^2+1)^ng & D(D^2+1)^ng\\
		D(D^2+1)^ng  & D^2(D^2+1)^{n}g 
	\end{array} \right|^{n+1}\ne 0	
\end{displaymath}
This proves Theorem \ref{wrf}, because Wronskians are nonzero.
\end{proof}
\begin{theorem}\label{lincom} 
Let $n$ be any fixed natural number and 
\begin{itemize}
	\item[(i)] $f(x)=x^n\sin x$, then the functions
	$$ f(x), Df(x), D^2f(x),\dots,D^{2n+1}f(x), D^{2n+2}f(x)$$
	\item[(ii)] $g(x)=x^n\cos x$, then the functions
	$$ g(x), Dg(x), D^2g(x),\dots,D^{2n+1}g(x), D^{2n+2}g(x)$$
\end{itemize}
are linearly dependent.
\end{theorem}
\begin{proof}
Let us denote 
\[\mbf{V}=
\left[\begin{array}{cccccc}
	\scriptstyle  f &\scriptstyle  Df &\scriptstyle  D^2f &\scriptstyle\dots& \scriptstyle D^{2n+1}f &\scriptstyle  D^{2n+2}f\\
	\scriptstyle Df &\scriptstyle D^2f &\scriptstyle D^3f & \scriptstyle \dots &\scriptstyle D^{2n+2}f  &\scriptstyle  D^{2n+3}f \\
	\scriptstyle D^2f &\scriptstyle D^3f &\scriptstyle D^4f & \scriptstyle\dots &\scriptstyle D^{2n+3}f &\scriptstyle D^{2n+4}f\\
	\scriptstyle \vdots &\scriptstyle \vdots  & \scriptstyle \vdots & \scriptstyle\ddots &\scriptstyle \vdots & \scriptstyle \vdots\\
	\scriptstyle D^{2n+1}f & \scriptstyle D^{2n+2}f& \scriptstyle D^{2n+3}f & \scriptstyle \dots & \scriptstyle D^{4n+2}f &\scriptstyle D^{4n+3}f\\
	\scriptstyle D^{2n+2}f & \scriptstyle D^{2n+3}f& \scriptstyle D^{2n+4}f & \scriptstyle \dots & \scriptstyle D^{4n+3}f &\scriptstyle D^{4n+4}f
\end{array} \right]\]
then
\begin{gather}
	\begin{array}{ll}
 W(f,Df,D^2f\dots, D^{2n+1}f, D^{2n+2}f)=|UVU^T|\\[0.5ex]
		&\\
		=\left|\begin{array}{cccccc}
			\scriptstyle  f &\scriptstyle  Df &\scriptstyle  (D^2+1)f &\scriptstyle \dots & \scriptstyle D(D^2+1)^nf &\scriptstyle  (D^2+1)^{n+1}f\\
			\scriptstyle Df & \scriptstyle D^2f& \scriptstyle D(D^2+1)f & \scriptstyle \dots & \scriptstyle D^2(D^2+1)^nf  &\scriptstyle  D(D^2+1)^{n+1}f \\
			\scriptstyle (D^2+1)f & \scriptstyle D(D^2+1)f & \scriptstyle (D^2+1)^2f & \scriptstyle \dots&\scriptstyle D(D^2+1)^{n+1}f&\scriptstyle (D^2+1)^{n+2}f\\
			\scriptstyle D(D^2+1)f & \scriptstyle D^2(D^2+1)f & \scriptstyle D(D^2+1)^2f & \scriptstyle \dots&\scriptstyle D^2(D^2+1)^{n+1}f&\scriptstyle D(D^2+1)^{n+2}f\\
			\scriptstyle \vdots &\scriptstyle \vdots&\scriptstyle \vdots &\scriptstyle \ddots & \scriptstyle \vdots  & \scriptstyle \vdots \\
			\scriptstyle D(D^2+1)^nf &\scriptstyle  D^2(D^2+1)^nf & \scriptstyle D(D^2+1)^{n+1}f &\scriptstyle \dots  & \scriptstyle D^2(D^2+1)^{2n}f & \scriptstyle D(D^2+1)^{2n+1}f\\
			\scriptstyle (D^2+1)^{n+1}f & \scriptstyle D(D^2+1)^{n+1}f& \scriptstyle (D^2+1)^{n+2}f& \scriptstyle \dots & \scriptstyle D(D^2+1)^{2n+1}f&\scriptstyle (D^2+1)^{2n+2}f
		\end{array} \right|\label{depend}\\
	&\\ =0
	\end{array}
\end{gather}
where \[U=U_{n+1,2n+3}U_{n,2n+3},\dots,U_{1,2n+3}\].\\
The Wronskian
\[ W(f,Df,D^2f\dots, D^{2n+1}f, D^{2n+2}f)=0\] because the last row and last column in (\ref{depend}) is null. Which proves that the functions $f,Df,D^2f\dots, D^{2n+1}f, D^{2n+2}f$ are linearly dependent. Similarly for the function $g$.
\end{proof}
\begin{comment}
Functions 
$$x^n\sin x, Dx^n\sin x, D^2x^n\sin x,\dots, D^{2n}x^n\sin x, D^{2n+1}x^n\sin x, D^{2n+2}x^n\sin x $$
resp.
$$x^n\cos x, Dx^n\cos x, D^2x^n\cos x,\dots, D^{2n}x^n\cos x, D^{2n+1}x^n\cos x, D^{2n+2}x^n\cos x$$
are linearly dependent. \par
It follows from the fact that functions\\
$x^n\sin x, Dx^n\sin x, D^2x^n\sin x,\dots, D^{2n}x^n\sin x, D^{2n+1}x^n\sin x$ are linearly independent, and that
\begin{gather}
\begin{array}{rl}
\mbf{V}=&\textrm{span}\{x^n\sin x, x^n\cos x, x^{n-1}\sin x, x^{n-1}\cos x, \dots, x\sin x, x\cos x,\sin x,\cos x\}\\
=&\textrm{span}\{x^n\sin x, Dx^n\sin x, D^2x^n\sin x,\dots, D^{2n+1}x^n\sin x\}\\[1ex]
&\textrm{dim}(\mbf{V})=2n+2
\end{array}\notag
\end{gather}
$B=\{x^n\sin x, Dx^n\sin x, D^2x^n\sin x,\dots, D^{2n+1}x^n\sin x\}$ is a basis of $\mbf{V}$
and $D^{2n+2}x^n\sin x\in \mbf{V}$ therefore  $D^{2n+2}x^n\sin x$ is a linear combination of\\ 
$x^n\sin x, Dx^n\sin x, D^2x^n\sin x,\dots, D^{2n+1}x^n\sin x.$\\
Similarly for the second sequence of functions.
\end{comment}
\begin{theorem}\label{wrfk}
	Let $n$ be any fixed natural number, $k$ any fixed non-negative integer, $f(x)=x^n\sin x$ (resp. $g(x)=x^n\cos x$) then the functions
\begin{center}{$D^kf(x), D^{k+1}f(x), D^{k+2}f(x),\dots,D^{2n+k}f(x), D^{2n+k+1}f(x)$} \\ 
	(resp. $D^kg(x), D^{k+1}g(x), D^{k+2}g(x),\dots,D^{2n+k}g(x), D^{2n+k+1}g(x)$)\\
\end{center}
	are linearly independent. 
\end{theorem}
\begin{proof}
If we substitute the function $D^kf$ in Theorem \ref{wr} in relations (\ref{eq:det1}) and (\ref{eq:det2}) instead of $f$, we get

\begin{gather}
\begin{array}{ll}
\scriptstyle W(D^kf,D^{k+1}f,D^{k+2}f,\dots,D^{2n+k}f,D^{2n+k+1}f)\\[0.5ex]
&\\
\scriptstyle=\left|\arraycolsep=0.5pt\def\arraystretch{1.2}
\begin{array}{cccccc}
	\scriptstyle  D^kf &\scriptstyle  D^{k+1}f &\scriptstyle  D^k(D^2+1)f &\scriptstyle \dots & \scriptstyle D^k(D^2+1)^nf &\scriptstyle  D^{k+1}(D^2+1)^nf\\
	\scriptstyle D^{k+1}f & \scriptstyle D^{k+2}f& \scriptstyle D^{k+1}(D^2+1)f & \scriptstyle \dots & \scriptstyle D^{k+1}(D^2+1)^nf  &\scriptstyle  D^{k+2}(D^2+1)^nf \\
	\scriptstyle D^k(D^2+1)f & \scriptstyle D^{k+1}(D^2+1)f & \scriptstyle D^k(D^2+1)^2f & \scriptstyle \dots&\scriptstyle D^k(D^2+1)^{n+1}f&\scriptstyle D^{k+1}(D^2+1)^{n+1}f\\
	\scriptstyle D^{k+1}(D^2+1)f & \scriptstyle D^{k+2}(D^2+1)f & \scriptstyle D^{k+1}(D^2+1)^2f & \scriptstyle \dots&\scriptstyle D^{k+1}(D^2+1)^{n+1}f&\scriptstyle D^{k+2}(D^2+1)^{n+1}f\\
	\scriptstyle \vdots &\scriptstyle \vdots&\scriptstyle \vdots &\scriptstyle \ddots & \scriptstyle \vdots  & \scriptstyle \vdots \\
	\scriptstyle D^k(D^2+1)^nf &\scriptstyle  D^{k+1}(D^2+1)^nf & \scriptstyle D^k(D^2+1)^{n+1}f &\scriptstyle \dots  & \scriptstyle D^k(D^2+1)^{2n}f & \scriptstyle D^{k+1}(D^2+1)^{2n}f\\
	\scriptstyle (D^2+1)^{n+1}f & \scriptstyle D(D^2+1)^{n+1}f& \scriptstyle (D^2+1)^{n+2}f& \scriptstyle \dots & \scriptstyle D^{k+1}(D^2+1)^{2n}f&\scriptstyle D^{k+2}(D^2+1)^{2n}f
\end{array} \right|\\
\end{array}\label{detT3}
\end{gather}

We have obviously (similarly to the proof of Theorem 2)
$$(D^2+1)f\ne 0, (D^2+1)^2f\ne 0,\dots, (D^2+1)^{n}f\ne 0$$
and
$$(D^2+1)^{n+1}f= 0, (D^2+1)^{n+2}f= 0,\dots, (D^2+1)^{2n}f= 0$$
\begin{equation}\label{eq:3trmp}
D^k(D^2+1)^nf\ne 0\  (\mathrm{resp.}\ D^{k+1}(D^2+1)^nf\ne 0, \ D^{k+2}(D^2+1)^nf\ne 0)
\end{equation}
since $(D^2+1)^{n}f\ne 0$ so also $D^k(D^2+1)^{n}f\ne 0$ 
(resp. $D^{k+1}(D^2+1)^{n}f\ne 0$, $D^{k+2}(D^2+1)^{n}f\ne 0$). Otherwise, $(D^2+1)^{n}f$ would have to be a polynomial of degree at least $(k-1)$ (resp. $k$, $(k+1)$), not equal to zero This is not possible because there is no such polynomial in the vector space (\ref{span}).\par
We will prove that
\begin{equation}\label{detT31}
\left|\begin{array}{cc}
 D^k(D^2+1)^nf & D^{k+1}(D^2+1)^nf\\
D^{k+1}(D^2+1)^nf  &  D^{k+2}(D^2+1)^nf
\end{array}
\right|\ne 0
\end{equation}
Suppose the converse statement holds that
$$D^k(D^2+1)^nf\cdot D^{k+2}(D^2+1)^nf-(D^{k+1}(D^2+1)^nf)^2=0$$
We put $y=D^k(D^2+1)^nf$, as we saw above $y\ne 0$. \par
We have the same differential equation as (\ref{DRT2})
whose general solution is (\ref{eq:drov}), 
which is not in the vector space (\ref{span})
only if $C_1=0$, but $y\ne 0$, so $C_1\ne 0$. \par
Similarly to the previous Theorem, we calculate the determinant (\ref{detT31}), which will be not equal to zero and then calculate the modified determinant (\ref{detT3}) with respect to (\ref{detT31}), which will be nonzero. We have thus proved Theorem \ref{wrfk}.
\end{proof}

\section{A method using the definition of linear (in)dependence of functions and combinatorial identities}

In this section we will study only the linear (in)dependence of functions 
\begin{equation}\label{zklpos}
	Df(x), D^2f(x), D^3f(x),\dots, D^{2n+1}f(x),D^{2n+2}f(x)
\end{equation}
where 
\begin{equation}\label{xnsin}
	f(x)=x^n\sin x\end{equation}
only using the definition of linear (in)dependence of functions (see introduction)
without using the Wronskian. Theorem \ref{wrfk} implies that the functions (\ref{zklpos}) are linearly independent. Even so, we will study this sequence of functions from a different point of view.\par
Obviously the Laplace transform does not give a satisfactory answer about the linear independence of these functions.
\begin{theorem} \label{Laplace}
	Denote $f(x)=x^n\sin x$ as above, where $n$ is any fixed natural number. The 
	functions 
	\begin{equation}\label{eq:hah}
		Df(x), D^2f(x), \dots, D^nf(x), D^{n+1}f(x)
	\end{equation}
	are linearly independent.
\end{theorem} 
\begin{proof} The functions (\ref{eq:hah}) are an exponential type and have all  derivatives for all values of $x\in \mathbb{R}$ and they are continuous. Then there exist Laplace transformations of $\mathcal{L}\{D^kf(t)\}$, for $k=1,2,\dots,n+1$. Denote  
$\mathcal{L}\{f(t)\}=F(s)$ then  
\[\begin{array}{l}
	\mathcal{L}\{Df(t)\}=sF(s)-f(0)=sF(s)\\
	\mathcal{L}\{D^2f(t)\}=s^2F(s)-sf(0)-Df(0)=s^2F(s)\\
	\qquad\vdots\\
	\mathcal{L}\{D^nf(t)\}=s^nF(s)-s^{n-1}f(0)-s^{n-2}Df(0)-\dots-D^{n-1}f(0)=s^nF(s)
	\\
	\mathcal{L}\{D^{n+1}f(t)\}=s^{n+1}F(s)-s^{n}f(0)-s^{n-1}Df(0)-\dots-D^nf(0)=s^{n+1}F(s)\\
	\textrm{where}\ f(0)=0,\ Df(0)=0,\ D^2f(0)=0,\ \dots,\ D^nf(0)= 0.
\end{array}	
\]
The functions 
$$sF(s),s^2F(s),\dots,s^{n+1}F(s)$$
are linearly independent. It follows from Lerch theorem \cite{SM} that the functions (\ref{eq:hah}) are linearly independent also. 
\end{proof}
\vspace{3mm}
\begin{remark}
Theorem \ref{Laplace} gives only incomplete information about the linear independence of the sequence of the derivatives of the function $x^n\sin x$.
\end{remark}
 Note that the first derivative of $f$ and each subsequent odd derivative of the function $f$ contain only members 
\[x^n\cos x,x^{n-1}\sin x,x^{n-2}\cos x,\dots, x\cos x,\sin x \]
in the linear combination.\newline
Similarly, the second derivative of $f$ and every other even derivative of this function contain only members 
\[x^n\sin x, x^{n-1}\cos x, x^{n-2}\sin x,\dots, x\sin x, \cos x\]
in the linear combination.
Therefore, it will be useful to choose (as we will see below) the basis  
\[B=\{x^n\cos x,x^n\sin x,x^{n-1}\sin x,x^{n-1}\cos x,
x^{n-2}\cos x,x^{n-2}\sin x,\dots,\sin x, \cos x\}.\] 
as the base of the vector space 
\begin{equation}\label{spanx}
V=\textrm{span}\{x^n\sin x, x^n\cos x, x^{n-1}\sin x, x^{n-1}\cos x, \dots, x\sin x, x\cos x,\sin x,\cos x\}
\end{equation}
Odd members of the basis $B$ (first, third, fifth, \ldots) are members odd derivatives of $f$; even members of the basis $B$ (second, fourth, sixth, \ldots) are members even derivatives of $f$ with decreasing powers of~$x$.\par  
All derivatives of the function $f$ belong to the vector space (\ref{spanx}).\par
We want to find out if there is at least one $\alpha_i\ne 0$ so that
\begin{equation}\label{eq:roa}
	\begin{array}{l}\alpha_1Df(x)+\alpha_2D^2f(x)+\dots+\alpha_kD^kf(x)+\dots\\
		+ \alpha_{2n+1}D^{2n+1}f(x)+\alpha_{2n+2}D^{2n+2}f(x)=0
	\end{array}
\end{equation}
where  $\alpha_i\in \mathbb{R},\,i=1,2,\dots,2n+2$ .\\ 
We express derivatives of the function $f$ in the basis $B$ (see \cite{F}) 
\begin{equation}\label{eq:roa1}
	\begin{array}{l}\alpha_1[Df]_B+\alpha_2[D^2f]_B+\dots+\alpha_k[D^kf]_B+\dots\\
		+ \alpha_{2n+1}[D^{2n+1}f]_B+
		\alpha_{2n+2}[D^{2n+2}f]_B=\mbf{0}
	\end{array}
\end{equation}
\begin{equation}\label{eq:roa2}
	\begin{array}{l}
\alpha_1(1,0,n,0,0,0,0,0,0,0,\dots)^T+
\alpha_2(0,-1,0,2n,0,(n)_{\underline{2}}, 0,0,0,0,\dots)^T\\
+\alpha_3(-1,0,-3n,0,3(n)_{\underline{2}},0,(n)_{\underline{3}},0,0,0,\dots)^T\\
+\alpha_4(0,1,0,-4n,0,-6(n)_{\underline{2}}, 0,4(n)_{\underline{3}},0,(n)_{\underline{4}}\dots)^T+\dots=\mbf{0}^T
	\end{array}
\end{equation}
The symbol T expresses the transpose of a vector.\par
 Or in the matrix form

\begin{displaymath}
	\begin{array}{rl}
		&\\[0.5ex]
		&\left[\begin{array}{ccccc}
			  1 & 0 & -1 &  0 
			& \dots \\
			 0 &  -1 &  0 &   1 
			& \dots \\
			 n &  0 &  -3n &  0 
			& \dots\\
			 0 & 2n & 0 & -4n 
			&  \dots\\
			 0 & 0 &  3(n)_{\underline{2}} 
			&  0 & \dots\\
			 0 &  (n)_{\underline{2}} &  0 
			&  -6(n)_{\underline{2}} & \dots\\
			 0 & 0 &  (n)_{\underline{3}} 
		    & 0	&\dots\\
			 0 &  0 &  0 
			&  4(n)_{\underline{3}} & \dots\\
			0 & 0 & 0 & 0 & \dots\\
			 0 &  0 &  0  
			&  (n)_{\underline{4}} &  \dots\\
			\vdots & \vdots & \vdots & \vdots &\ddots\\
		\end{array} \right]
	\left[\begin{array}{c}
		\alpha_1\\
		\alpha_2\\
		\alpha_3\\
		\vdots\\ 
		\alpha_{2n+2}
	\end{array} \right]=
\left[\begin{array}{c}
	0\\
	0\\
	0\\
	\vdots\\ 
	0
\end{array} \right]
\\
		&\\
		& =\mbf{A}_{(2n+2)\times(2n+2)}[a_{ij}]\mbf{\alpha}=\mbf{0}
	\end{array}
\end{displaymath}
where we denoted
$(x)_{\underline{n}}$ as the falling factorial defined by following
\footnote{We use the notation $(x)_n$ for the Pochhammer symbol.}
$$(x)_{\underline{n}}=\left\{ \begin{matrix} 1 & \textrm{for}\,n=0\\
	x(x-1)(x-2)\dots(x-n+1) & \textrm{for}\,n\ge 1\\	
\end{matrix} \right.$$
and
\begin{equation}\label{eq:haa}
	a_{ij}=\displaystyle{\frac{1+(-1)^{i+j}}{2}\left(-1\right)^{[\frac{2j+1}4]+
	[\frac{2i+1}8]}{j \choose [\frac{2i-1}4]}(n)_{\underline{[\frac{2i-1}4]}}}
\end{equation}
\[\mbf{\alpha}= \left[\begin{array}{c}
	\alpha_1\\
	\alpha_2\\
	\alpha_3\\
	\vdots\\ 
	\alpha_{2n+2}
\end{array} \right], \qquad
\mbf{0}=\left[\begin{array}{c}
	0\\
	0\\
	0\\
	\vdots\\ 
	0
\end{array} \right] \]\par
Using mathematical software and formula (\ref{eq:haa}) we generated the matrix $\mbf{A}$ for $f(x)=x^4\sin x$.
\[\mbf{A}=
\left[\begin{array}{cccccccccc}
 1 &  0 & -1 & 0 & 1 & 0  & -1 &  0 & 1 & 0 \\
  0 &  -1 & 0 & 1 &  0 &  -1  & 0 &  1 &  0 &  -1 \\
  4 &  0 & -12 & 0 &  20 &  0  & -28 &  0 &  36 &  0 \\
  0 &  8 & 0 & -16 &  0 &  24  & 0 &  -32 &  0 & 40 \\
  0 &  0 & 36 & 0 &  -120 &  0  & 252 &  0 &  -432 &  0 \\
  0 &  12 & 0 & -72 &  0 &  180  & 0 & -336 & 0 & 540 \\
  0 &  0 & 24 & 0 &  -240 &  0  & 840 &  0 & -2016 &  0 \\
  0 &  0 & 0 & 96 &  0 & -480  & 0 & 1344 & 0 & -2880 \\
  0 &  0 & 0 & 0 &  120 &  0  & -840 &  0 & 3024 &  0 \\
  0 &  0 & 0 & 24 &  0 & -360  & 0 & 1680 & 0 & -5040\\
\end{array} \right]\]

The construction of the elements $a_{ij}$ of the matrix $\mbf{A}$ is interesting.\newline
The member $\frac{1+(-1)^{i+j}}{2}$ of (\ref{eq:haa}) generates 0 in the matrix $\mbf{A}$ if i+j is odd, otherwise~1. \newline
The member 
$\left(-1\right)^{[\frac{2j+1}4]+[\frac{2i+1}8]}$ ensures the correct sign of the derivatives of function $f$ in the matrix $A$.
\newline
The most interesting is the construction of the last term 
${{j \choose [\frac{2i-1}4]}}
(n)_{\underline{[\frac{2i-1}4]}}$
in (\ref{eq:haa}), which follows from the Leibniz formula for the higher derivatives of the product of two functions. 

Now we perform some equivalent transformations (which do not change the rank of the matrix) on the matrix $\mbf{A}$.\\
Multiplying the $i$-th row of the matrix $\mbf{A}$ (for all 
$i = 1,2,\dots,2n+2$) with the number $$\displaystyle\frac{1}{(-1)^{[\frac{2i+1}8]}(n)_{\underline{[\frac{2i-1}4]}}}$$
we get the matrix $\mbf{A}^{\prime}$ equivalent to the matrix $\mbf{A}$
\begin{equation}\label{eq:hab}
	\begin{array}{c}
		\mbf{A}^\prime=[a'_{ij}]_{(2n+2)\times(2n+2)}\\[2ex]
		a^\prime_{ij}=\frac{1+(-1)^{i+j}}{2}\left(-1\right)^{[\frac{2j+1}4]}
		\displaystyle{j \choose [\frac{2i-1}4]}
	\end{array}
\end{equation}

And finally by multiplying the $j$-th row of the matrix $\mbf{A}^\prime$ (for all\newline
$j~=~1,2,\dots,2n+2$) with the number 
$$\frac{1}{\left(-1\right)^{[\frac{2j+1}4]}}$$
we get the matrix $\mbf{A}^{\prime\prime}$ equivalent to the matrix $\mbf{A}^{\prime}$
\begin{equation}\label{eq:hac}
	\begin{array}{c}
		\mbf{A}^{\prime\prime}=[a^{\prime\prime}_{ij}]_{(2n+2)\times(2n+2)}\\[2ex]
		a^{\prime\prime}_{ij}=\frac{1+(-1)^{i+j}}{2}
		\displaystyle {j \choose [\frac{2i-1}{4}]}
	\end{array}
\end{equation}
The matrix $\mbf{A}^{\prime\prime}$ has the form 

\begin{equation}\label{had}
	\mbf{A}^{\prime\prime}=\begin{bmatrix}
		{1\choose 0}&0&{3\choose 0}&0&{5\choose 0}&0&\dots& 
		{2n+1\choose 0}&0\\[1ex]
		0&{2\choose 0}&0&{4\choose 0}&0&{6\choose 0}&\dots&0& 
		{2n+2\choose 0}\\[1ex]
		{1\choose 1}&0&{3\choose 1}&0&{5\choose 1} &0&\dots&
		{2n+1\choose 1}&0&\\[1ex]
		0&{2\choose 1}&0&{4\choose 1}&0&{6\choose 1}&\dots&0& 
		{2n+2\choose 1}\\[1ex]
		{1\choose 2}&0&{3\choose 2}&0&{5\choose 2} &0&\dots&
		{2n+1\choose 2}&0&\\[1ex]
		\vdots&\vdots&\vdots&\vdots&\vdots&\vdots&\ddots&\vdots&\vdots\\[1ex]
		{1\choose n}&0&{3\choose n}&0&{5\choose n}&0&\dots&
		{2n+1\choose n}&0&\\[1ex]
		0&{2\choose n}&0&{4\choose n}&0&{6\choose n}&
		\dots&0&{2n+2\choose n}\\[1ex]
	\end{bmatrix}
\end{equation}
We need to prove that the matrix (\ref{had}) is regular.
\\
\begin{lemma}\label{ED}
\begin{equation}\label{eq:a}
	\sum_{k=1}^n\left(-\frac12\right)^{n-k}{2j-1\choose k-1}
	{2n-k-1\choose n-1}=2^{n-1}{j-1\choose n-1} 
\end{equation}
\end{lemma}
\begin{proof}
The method used to prove the combinatorial identity (\ref{eq:a}) was described in the article \cite{D}. The method is used to calculate summations in which binomials, Gamma functions and possibly Pochhammer symbols occur. The summation can be either finite or infinite. The characteristic of the method is to write the summation as a hypergeometric function. If the final hypergeometric function is known, then the result follows. Combinatorial identities can often be proven with this method.\\
We start by rewriting (\ref{eq:a}) as a summation from $k=0$ to $k=n-1$. This gives
\begin{equation}\label{eq:aa1}
	\sum_{k=0}^{n-1}\left(-\frac12\right)^{n-k-1}{2j-1\choose k}
	{2n-k-2\choose n-1}=2^{n-1}{j-1\choose n-1} 
\end{equation}

We convert the binomials to a quotient with a Pochhammer symbol by using
\begin{displaymath}
{a\choose i}=(-1)^i\frac{(-a)_i}{i!}
\end{displaymath}
Most of the used properties of the Gamma functions in combination with Pochhammer symbols are mentioned in appendix I of \cite{PBM}. Application to the left hand-side $L$ of (\ref{eq:aa1}) gives
\begin{displaymath}
L=\sum_{k=0}^{n-1}(-1)^k\frac{(1-2j)_k}{k!}(-1)^{n-1}\frac{(k+2-2n)_{n-1}}{(n-1)!}
\end{displaymath}
For the evaluation of the numerator in the second fraction we use\\
$\Gamma(a+k)=\Gamma(a)(a)_k$ and 
$\frac{\Gamma(-m)}{\Gamma(-n)}=(-1)^{n-m}\frac{\Gamma(n+1)}{\Gamma(m+1)}.$ \\
Application gives
\begin{displaymath}
	\begin{array}{rl}
(k+2-2n)_{n-1}&=\dfrac{\Gamma(k+1-n)}{\Gamma(k+2-2n)}=\dfrac{\Gamma(1-n)}{\Gamma(2-2n)}
\dfrac{(1-n)_k}{(2-2n)_k}\\
&=(-1)^{n-1}\dfrac{\Gamma(2n-1)}{\Gamma(n)}\dfrac{(1-n)_k}{(2-2n)_k}
	\end{array}
\end{displaymath}
Application to $L$ gives after some manipulations 
\[L=\left(-\frac12\right)^{n-1}\frac{\Gamma(2n-1)}{\Gamma(n)^2}\sum_{k=0}^{n-1}
\frac{(1-n)_k(1-2j)_k}{(2-2n)_k}\frac1{k!}2^k\]
The summation can be written as a hypergeometric function in the next formula
\[L=\left(-\frac12\right)^{n-1}\frac{\Gamma(2n-1)}{\Gamma(n)^2}\]
\cite[7.3.1(149)]{PBM} gives
\[{}_2F_1\left(\begin{matrix}-m,b\\
	-2m\end{matrix};x\right)= \frac{m!}{(2n)!}\Gamma(1-b)(-x)^m(x-1)^{-(m-b)/2}P_m^{m+b}\left(1-\frac2x\right) \]
$P_{\nu}^{\mu}$ is the associate Legendre function of the first kind. Properties of these functions can be found for example in \cite{GR}. Application gives
\[{}_2F_1\left(\begin{matrix}1-n,1-2j\\
2-2n\end{matrix};2\right)=\frac{(n-1)!}{(2n-2)!}\Gamma(2j)(-2)^{n-1}P_{n-1}^{n-2j}(0)\]
After some simplification we  get for  $L$
\[L=\frac1{\Gamma(n)}\Gamma(2j)P_{n-1}^{n-2j}(0)\]
Using the duplication formula of Legendre for the Gamma function gives
\[L=\frac1{\Gamma(n)}\dfrac{\Gamma(j)\Gamma\left(j+\frac12\right)2^{2j-1}}{\sqrt{\pi}}
P_{n-1}^{n-2j}(0)\]
For the associate Legendre function with zero argument we use \cite[8.756(1)]{GR}:
\[P_{\nu}^{\mu}(0)=\dfrac{2^\mu\sqrt{\pi}}{\Gamma\left(\frac{2-\mu+\nu}{2}\right)
		\Gamma\left(\frac{1-\mu-\nu}{2}\right)}\]
Application gives
\[P_{n-1}^{n-2}(0)=\dfrac{2^{n-2j}\sqrt{\pi}}{\Gamma\left(j-\frac12\right)\Gamma(j-n+1)}\]
After some simplification we get at last
\begin{displaymath}
\begin{array}{rl}
L&=\dfrac1{\Gamma(n)}\dfrac{\Gamma(j)\Gamma\left(j+\frac12\right)2^{2j-1}}{\sqrt{\pi}}
\dfrac{2^{n-2j}\sqrt{\pi}}{\Gamma\left(j+\frac12\right)\Gamma(j-n+1)}\\
&=2^{n-1}\dfrac{\Gamma(j)}{\Gamma(n)\Gamma(j-n+1)}=2^{n-1}{j-1\choose n-1}
\end{array}
\end{displaymath}
\end{proof}

\begin{lemma}\label{Det1}
 Let us create a minor \(\mbf{B}\) of $\mbf{A}^{\prime\prime}$ \cite{WE3}
\[\mbf{B}=\mbf{A}^{\prime\prime}_
{\substack {2,4,\dots,2n+2 \\ 2,4,\dots,2n+2}}
=\left[2j-1\choose i-1\right]_{(n+1)\times (n+1)}\]
then 
\begin{equation}\label{hadx}	
|\mbf{B}|=
\begin{vmatrix} 
		{1\choose 0}&{3\choose 0}&{5\choose 0}&\dots& 
		{2n+1\choose 0}\\[1ex]
		{1\choose 1}&{3\choose 1}&{5\choose 1}&\dots& 
		{2n+1\choose 1}\\[1ex]
		{1\choose 2}&{3\choose 2}&{5\choose 2}&\dots& 
		{2n+1\choose 2}\\[1ex]
		\vdots&\vdots&\vdots&\ddots& \vdots\\[1ex]
		{1\choose n}&{3\choose n}&{5\choose n}&\dots& 
		{2n+1\choose n}\\[1ex]
		\end{vmatrix}=2^{n+1 \choose 2 }
\end{equation}
\end{lemma}
\begin{proof}
We define the matrix	
	\begin{equation}\label{eq:Tnn}
		\mbf{T}=[t_{i,j}]_{(n+1)\times (n+1)}, \  \textrm{where}\ 
		t_{i,j}= \left\{ \begin{array}{cl}
			(-\frac{1}{2})^{(i-j)}{2i-j-1 \choose i-j} & \textrm{if $j\le i$}\\[0.5ex]
			0 & \textrm{otherwise}
		\end{array} \right.
	\end{equation} 
	The matrix $\mbf{T}$ is a lower triangular matrix with the determinant $|\mbf{T}|=1.$
	Let us denote 
		$$\mbf{T}\mbf{B}=\mbf{G}$$
	Then	
\begin{displaymath}
\begin{array}{rl}
g_{ij}&=\displaystyle\sum_{k=1}^{n+1}t_{ik}b_{kj}=\displaystyle\sum_{k=1}^{i}t_{ik}b_{kj}\\
&=\displaystyle\sum_{k=1}^{i}\left(-\frac12\right)^{i-k}{2i-k-1 \choose i-k}{2j-1 \choose k-1}=2^{i-1}{j-1 \choose i-1}
\end{array}
\end{displaymath}
	
	Then using (\ref{eq:a}) we get
	$$|\mbf{B}|=|\mbf{T}\mbf{B}|=|\mbf{G}|$$
	where
	\begin{equation}\label{eq:Cnn}
		\mbf{G}=[g_{i,j}]_{(n+1)\times (n+1)}, \quad  g_{i,j}=2^{i-1}{j-1\choose i-1}
	\end{equation}
	The matrix $\mbf{G}$ is an upper diagonal matrix with elements on the diagonal:
	$$2^0,2^1,2^2,\dots,2^{n-1},2^n$$
	So 
	$$|\mbf{B}|=|\mbf{G}|=2^0\cdot2^1\cdot2^2\cdots 2^{n-1}\cdot 2^n=2^{n+1 \choose 2}$$
\end{proof}
\bigskip
%\begin{comment}
\begin{lemma}\label{Det2}
Let us create a minor \(\mbf{C}\) of $\mbf{A}^{\prime\prime}$ \cite{WE3}
\[\mbf{C}=\mbf{A}^{\prime\prime}_
{\substack{1,3,\dots,2n+1 \\ 1,3,\dots,2n+1}}
=\left[2j\choose i-1\right]_{(n+1)\times (n+1)}\]
then
\begin{equation}\label{hae}
|\mbf{C}|=\begin{vmatrix}
		{2\choose 0}&{4\choose 0}&{6\choose 0}&\dots&{2n+2\choose 0}\\[1ex]
		{2\choose 1}&{4\choose 1}&{6\choose 1}&\dots&{2n+2\choose 1}\\[1ex]
		{2\choose 2}&{4\choose 2}&{6\choose 2}&\dots&{2n+2\choose 2}\\[1ex]
		\vdots&\vdots&\vdots&\ddots&\vdots\\
		{2\choose n}&{4\choose n}&{6\choose n}&\dots&{2n+2\choose n}\\
\end{vmatrix}=2^{n+1\choose 2 }
\end{equation}
\end{lemma}
\begin{proof}
We now define the matrix \mbf{T} as
	\begin{displaymath}
			\mbf{T}=[t_{i,j}]_{(n+1)\times (n+1)}, \quad 
				t_{i,j}=\left\{ \begin{array}{cl}
					1 & \textrm{if $i=j$ or $i=j+1$} \\[0.5ex]
					0 & \textrm{otherwise} \end{array} \right.
	\end{displaymath}
The determinant $|\mbf{T}|=1$. 
Using
		$${2j-1 \choose i-2} + {2j-1 \choose i-1}={2j \choose i-1}$$
gives
\begin{displaymath}
\begin{array}{rl}
\mbf{C}&=\mbf{T}\mbf{B}\\
|\mbf{C}|&=|\mbf{T}|\cdot|\mbf{B}|=1\cdot|\mbf{B}|=2^{n+1\choose 2}
\end{array} 
\end{displaymath}
\end{proof} 

\begin{lemma}\label{expn} Let $a_{ij}, b_{ij}\in \mathbb{R}$, $i=1,2,\dots,n$, then
\begin{equation}\label{eq:fnl}
	\begin{array}{rl}
		\left|\mbf{H}\right|&=\left|\begin{array}{cccccccc}
			a_{11} & 0      & a_{12} & 0      & a_{13} &\ldots  & a_{1n} & 0     \\
			0      & b_{11} & 0      & b_{12} & 0      & \ldots & 0      & b_{1n}\\
			a_{21} & 0      & a_{22} & 0      & a_{23} &\ldots  & a_{2n} & 0     \\
			0      & b_{21} & 0      & b_{22} & 0      & \ldots & 0      & b_{2n}\\
			a_{31} & 0      & a_{32} & 0      & a_{33} &\ldots  & a_{3n} & 0     \\
			\vdots & \vdots & \vdots &\vdots  & \vdots & \ddots &\vdots  & \vdots \\
			a_{n,1} & 0     & a_{n2} & 0      & a_{n3} &\ldots  & a_{nn} & 0     \\
			0      & b_{n1} & 0      & b_{n2} & 0      & \ldots & 0      & b_{nn}			
		\end{array} \right|\\ 
		& \\
		&=\left|\begin{array}{cccc}
			a_{11}  & a_{12} &\ldots  & a_{1n}  \\
			a_{21}  & a_{22} &\ldots  & a_{2n} \\
			\vdots  & \vdots &\ddots  & \vdots \\
			a_{n1}  & a_{n2} &\ldots  & a_{2n}
		\end{array} \right|
		\left|\begin{array}{cccc}
			b_{11}  & b_{12} &\ldots  & b_{1n}  \\
			b_{21}  & b_{22} &\ldots  & b_{2n} \\
			\vdots  & \vdots &\ddots  & \vdots \\
			b_{n1}  & b_{n2} &\ldots  & b_{2n}
		\end{array} \right|
	\end{array} 
\end{equation}
\end{lemma}
\begin{proof}
The determinant $\left|\mbf{H}\right|$ can be computed by using the Laplace generalized  expansion along the first, third, fifth, ..., ($2n-1$)-th rows. \par 
Let $n$ be an odd number (analogous if $n$ be  an even number). Then
\begin{displaymath}
		\left|\mbf{H}\right|=\left|\begin{array}{ccccccc}
			a_{11}  & 0      & a_{12} & 0      & \ldots  & 0     & a_{1,\frac{n+1}2}\\
			a_{2,1} & 0      & a_{22} & 0      & \ldots  & 0     & a_{2,\frac{n+1}2}\\
			\vdots  & \vdots & \vdots &\vdots  & \ddots  &\vdots & \vdots \\
			a_{n,1} & 0     & a_{n2}  & 0      & \ldots  & 0     & a_{n,\frac{n+1}2} 
		\end{array} \right|\cdot
		\left|\begin{array}{ccccc}
			b_{1,\frac{n+1}2}  & 0      &\ldots & 0       & b_{1n}  \\
			b_{2,\frac{n+1}2}  & 0      &\ldots & 0       & b_{2n} \\
			\vdots             & \vdots &\ddots & \vdots  & \vdots \\
			b_{n,\frac{n+1}2}  & 0      &\ldots & 0       & b_{nn} \\
		\end{array} \right|\\
	\end{displaymath}
\begin{displaymath}\hspace{12mm}
		-\left|\begin{array}{cccccc}
			a_{11}  & 0      & \ldots & a_{1,\frac{n-1}2} & 0    & 0\\
			a_{21}  & 0      & \ldots & a_{2,\frac{n-1}2} & 0    & 0\\
			\vdots  & \vdots & \ddots &\vdots             & \vdots & \vdots\\
			a_{n1}  & 0      & \ldots & a_{n,\frac{n-1}2} & 0    & 0
		\end{array} \right|\cdot
		\left|\begin{array}{cccccc}
			0  & 0      & b_{1,\frac{n+3}2} & 0     & \ldots & b_{1n} \\
			0  & 0      & b_{2,\frac{n+3}2} & 0     & \ldots & b_{2n} \\
			\vdots & \vdots & \vdots            &\vdots & \ddots & \vdots\\
			0  & 0      & b_{n,\frac{n+3}2} & 0     & \ldots & b_{nn} \\
		\end{array} \right| + \ldots 
\end{displaymath}
\begin{displaymath}\hspace{-10mm}
+\left|\begin{array}{cccc}
	a_{11}  & a_{12}  & \ldots & a_{1,n}\\
	a_{21}  & a_{22}  & \ldots & a_{2,n}\\	
	\vdots  & \vdots  & \ddots &\vdots  \\
	a_{n1}  & a_{n2}  & \ldots & a_{n,n}
\end{array} \right|\cdot
\left|\begin{array}{cccc}
	b_{11}  & b_{12}  & \ldots & b_{1,n}\\
	b_{21}  & b_{22}  & \ldots & b_{2,n}\\	
	\vdots  & \vdots  & \ddots &\vdots  \\
	b_{n1}  & b_{n2}  & \ldots & b_{n,n}
\end{array} \right|  + \ldots\\
\end{displaymath}
\begin{displaymath}\hspace{-9mm}
	- \left|\begin{array}{ccccc}
	0   & a_{1\frac{n+3}2}  & \ldots & a_{1,n} & 0\\ 
	0   & a_{2\frac{n+3}2}  & \ldots & a_{2,n} & 0\\	
	\vdots  & \vdots            & \ddots &\vdots   & \vdots \\
	0   & a_{n\frac{n+3}2}  & \ldots & a_{n,n} & 0
\end{array} \right|\cdot
\left|\begin{array}{cccc}
	b_{11}  & b_{12}  & \ldots & b_{1,\frac{n+1}2}\\
	b_{21}  & b_{22}  & \ldots & b_{2,\frac{n+1}2}\\	
	\vdots  & \vdots  & \ddots &\vdots  \\
	b_{n1}  & b_{n2}  & \ldots & b_{n,\frac{n+1}2}
\end{array} \right| \\
\end{displaymath}
\begin{displaymath}\hspace{-21mm}
	= \left|\begin{array}{cccc}
	a_{11}  & a_{12}  & \ldots & a_{1,n}\\
	a_{21}  & a_{22}  & \ldots & a_{2,n}\\	
	\vdots  & \vdots  & \ddots &\vdots  \\
	a_{n1}  & a_{n2}  & \ldots & a_{n,n}
\end{array} \right| \cdot
\left|\begin{array}{cccc}
	b_{11}  & b_{12}  & \ldots & b_{1,n}\\
	b_{21}  & b_{22}  & \ldots & b_{2,n}\\	
	\vdots  & \vdots  & \ddots &\vdots  \\
	b_{n1}  & b_{n2}  & \ldots & b_{n,n}
\end{array} \right|
\end{displaymath}
All pairs of determinants in the calculation, except one (the last), have an identical zero value.
\end{proof}

\begin{theorem}\label{fincons} It follows from Lemma \ref{expn} that
	\begin{equation}\label{eq:con}
	|\mbf{A}^{\prime\prime}|=|\mbf{B}|\cdot|\mbf{C}|=2^{n+1 \choose 2}\cdot 2^{n+1 \choose 2}=2^{n(n-1)}\ne 0
	\end{equation}
which proves that the matrix $\mbf{A}^{\prime\prime}$ (see \ref{had}) is regular.
\end{theorem}
\begin{conclusion} \{The consequence of Theorem \ref{fincons}.\}\par
The functions
$$Dx^n\sin x, D^2x^n\sin x, D^3x^n\sin x,\ldots, D^{2n+2}x^n\sin x$$ 
are linearly independent.
\end{conclusion}
\section{Apendix}
We have defined  $n\times n$ matrices $R_k$, for $k=1,2,\dots,n-1$ in Lemma \ref{defrk}.  In the proof of Lemma \ref{lmm1}, we used the matrix product $R_kR_{k-1}\cdots R_2R_1$. We are interested in what type such matrices will have. This is shown in the following theorem.

\begin{theorem}{(Pascal triangle in the matrix.)}\ \vspace{1mm}
\begin{displaymath}
\begin{array}{l}
\mbf{R}_{n-1}\mbf{R}_{n-2}\dots\mbf{R}_2\mbf{R}_1\\[0.7ex]
=\left[ \begin{array}{cccccc}   
\displaystyle{0 \choose 0} &      0       & 0        &\dots &0        & 0\\[1ex]
\displaystyle{1 \choose 0} &\displaystyle{1 \choose 1} & 0            & \dots &0 & 0\\[1ex]
\displaystyle{2 \choose 0} &\displaystyle{2 \choose 1} &\displaystyle{2 \choose 2} &\dots    &0   & 0\\[1ex]
\vdots        &\vdots        &\vdots        & \ddots &\vdots        &\vdots\\[1ex]
\displaystyle{n-2\choose 0}&\displaystyle{n-2\choose 1}&\displaystyle{n-2\choose 2}&\dots&\displaystyle{n-2\choose n-2} &0 \\[1ex]
\displaystyle{n-1\choose 0}&\displaystyle{n-1\choose 1}&\displaystyle{n-1\choose 2}&\dots&\displaystyle{n-1\choose n-2} &\displaystyle{n-1\choose n-1}
\end{array}\right]
\end{array}
\end{displaymath}	
This can be written in the form
\begin{equation}\label{d1}
	\left[i-1\choose j-1\right]_{n\times n}
\end{equation}
\begin{proof}Let us use mathematical induction. 
For $R_1$ we have the matrix 
\begin{displaymath}
	\begin{array}{l}
	\mbf{R}_1=\left[ \begin{array}{cccccc}   
			1 & 0 & 0 &\dots  &0 & 0\\[1ex]
			1 & 1 & 0 & \dots &0 & 0\\[1ex]
			0 & 1 & 1 &\dots  &0 & 0\\[1ex]
			\vdots&\vdots&\vdots& \ddots &\vdots&\vdots\\[1ex]
			0 & 0 & 0 &\dots& 1 &0 \\[1ex]
			0 & 0 & 0 &\dots& 1 & 1
		\end{array}\right]
	\end{array}
\end{displaymath}	
	or in the form
\begin{equation}\label{d2}
\mbf{R}_1=\left[1\choose j-i+1\right]_{n\times n}
\end{equation}
where we used the definition $n\choose k$  according to Kronenburg \cite{KMJ} also in the following cases. The matrix is a lower triangular matrix in which the $i$-th row ($i=2,3,..,n$) starts with $i-2$ zeros, followed by $1\choose 0$, $1\choose 1$.\par 
Let's see how the matrix $R_2R_1$ will look like. After multiplying the matrices using the formula ${n\choose k}+{n\choose k+1}={n+1\choose k+1}$, we get  the matrix, where the first two rows remain unchanged. Next rows for ($i=3,4,\dots,n$) have the entries 
$2\choose j-i+2$
\begin{displaymath}
	\begin{array}{l}
	\mbf{R}_2\mbf{R}_1=\left[ \begin{array}{cccccccc}   
			1 & 0 & 0 & 0 &\dots & 0 & 0 & 0 \\[1ex]
			1 & 1 & 0 & 0 &\dots & 0 & 0 & 0 \\[1ex]
			1 & 2 & 1 & 0 &\dots & 0 & 0 & 0 \\[1ex]
			0 & 1 & 2 & 1 &\dots & 0 & 0 & 0 \\[1ex]
			\vdots&\vdots&\vdots& \vdots& \ddots &\vdots&\vdots&\vdots\\[1ex]
			0 & 0 & 0 & 0 & \dots& 2& 1 & 0\\[1ex]
			0 & 0 & 0 & 0 &\dots & 1 & 2 & 1
		\end{array}\right]
	\end{array}
\end{displaymath}	
Assume that the product $R_kR_{k-1}\dots R_2R_1$ will have an analogous structure, i.e. that
entries of $i$-th row for $i=1,2,\dots,k+1$ have the form ${i-1\choose j-1}$ and the next ones the form ${k\choose j-i+k}$ for $i=k+2,k+3,\dots,n$. The $k+1$ rows are the same as in the matrix (\ref{d1}).\par 
Let's compute $R_{k+1}(R_kR_{k-1}\dots R_2R_1)$. We get a matrix that will have $k+1$ rows the same to the $R_kR_{k-1}\dots R_2R_1$ matrix. The next rows will have the form ${k+1\choose j-i+k+1}$. Which proves the assumed structure of the matrix $R_{k+1}(R_kR_{k-1}\dots R_2R_1)$. 
For $R_{n-1}R_{n-2}\dots R_2R_1$ we get the matrix (\ref{d1}).
\end{proof}
\end{theorem}
In (\ref{hadx}) and \ref{hae}) it was proved that
$$|\mbf{B}|=|\mbf{C}|=2^{n+1\choose 2}. $$
The next two Theorems generalize the relations (\ref{hadx}), \ref{hae}) and give much more general results.
\begin{lemma}\label{rzb}
\begin{displaymath}
	\begin{array}{l}
		\left| \begin{array}{cccc}   
\displaystyle{x_1\choose 0} & \displaystyle{x_2\choose 0} &\dots &\displaystyle{x_n\choose 0}\\[1ex]
\displaystyle{x_1\choose 1} & \displaystyle{x_2\choose 1} &\dots &\displaystyle{x_n\choose 1}\\[1ex]
\displaystyle{x_1\choose 2} & \displaystyle{x_2\choose 2} &\dots &\displaystyle{x_n\choose 2}\\[1ex]
\vdots         & \vdots         &\ddots& \vdots       \\[1ex]
\displaystyle{x_1\choose n-1} & \displaystyle{x_2\choose n-1} &\dots &\displaystyle{x_n\choose n-1}\\
\end{array}\right|=\dfrac{1}{1!2!\cdots (n-1)!}\displaystyle\prod_{1\le i<j\le n}(x_j-x_i)
	\end{array}
\end{displaymath}	
\begin{proof}
Remove from the third line $\dfrac1{2!}$, from the fourth $\dfrac1{3!}$, etc. We have
\begin{equation}\label{wtran}
	\begin{array}{l}
\dfrac{1}{1!2!\cdots (n-1)!}\left| \begin{array}{cccc}   
1 & 1 &\dots &1\\[1ex]
x_1 & x_2 &\dots &x_n\\[1ex]
\displaystyle{x_1(x_1-1)} & \displaystyle{x_2(x_2-1)} &\dots & \displaystyle{x_n(x_n-1)}\\[1ex]
\vdots         & \vdots         &\ddots& \vdots       \\[1ex]
\displaystyle{\prod_{k=0}^{n-2}(x_1-k)} & 
\displaystyle{\prod_{k=0}^{n-2}(x_2-k)} &\dots &
\displaystyle{\prod_{k=0}^{n-2}(x_n-k)}\\
\end{array}\right|
	\end{array}
\end{equation}	
We want to show that the matrix of the determinant (\ref{wtran}) can be transformed into a transposed Vandermonde matrix by using elementary row operations.\par 
$1^{\mathrm{st}}$   step: we add the second row of the matrix of the determinant to the third. We get a matrix whose first three rows are the same as in the corresponding transposed Vandermonde matrix.\par 
$2^{\mathrm{nd}}$ step: multiply the second row by $-s(3,1)$, the third row by $s(3,2)$ and add to the fourth row. We get a matrix of the determinant whose first four rows are the same as in the corresponding Vandermonde matrix. The $-s(3,1)$, $s(3,2)$ are Stirling numbers of the first kind.\par 
Generally:\par
$k^{\mathrm{th}}$ step: we multiply the second row by $(-1)^{k+1}s(k+1,1)$, the third by $(-1)^{k}s(k+1)$, $\dots$, $(k+1)^{\mathrm{th}}$ we multiply by $s(k+1,k)$ and add all to\linebreak $(k+2)^{\mathrm{th}}$ row. We get a matrix whose first $k+2$ rows are the same as in the corresponding Vandermonde matrix.
After applying $n-2$ such steps, we get the determinant of the transposed Wronskian matrix, namely
\begin{equation}\notag 
	\begin{array}{l}
		\dfrac{1}{1!2!\cdots (n-1)!}\left| \begin{array}{cccc}   
			1 & 1 &\dots &1\\[1ex]
			x_1 & x_2 &\dots &x_n\\[1ex]
			\displaystyle{x_1^2} & \displaystyle{x_2^2} &\dots & \displaystyle{x_n^2}\\[1ex]
			\vdots         & \vdots         &\ddots& \vdots       \\[1ex]
			\displaystyle{x_1^{n-1}} & \displaystyle{x_2^{n-1}} &\dots &
			\displaystyle{x_n^{n-1}}\\
		\end{array}\right|=\dfrac{1}{1!2!\cdots (n-1)!}\displaystyle\prod_{1\le i<j\le n}(x_j-x_i)
	\end{array}
\end{equation}	
We used the well-known relation for the value of the Vandermonde determinant.
\end{proof}	
\end{lemma}
\begin{theorem} The determinant
\begin{displaymath}
	\left|aj+b\choose i-1\right|_{n\times n}=a^{n\choose 2}
\end{displaymath}
where $a, b$ are any real numbers.
\begin{proof} It follows from Lemma \ref{rzb} that
\begin{displaymath}
\begin{array}{rl}
\left|aj+b\choose i-1\right|_{n\times n} &=\dfrac{1}{1!2!\cdots (n-1)!}(an-a(n-1))(an-a(n-2))\cdots (an-a)\\
	&\cdot (a(n-1)-a(n-2))(a(n-1)-a(n-3))\cdots (a(n-1)-a)\\
	&\cdots (3a-2a)(3a-a)(2a-a)\\
	&=\dfrac{1}{1!2!\cdots (n-1)!}a^{n-1}1\cdot2\cdots(n-1)a^{n-2}1\cdot2\cdots(n-2)\cdots a\cdot1\\
	&=\dfrac{1}{1!2!\cdots (n-1)!}(n-1)!(n-2)!\cdots2!\cdot1!a^{1+2+3+\dots+(n-1)}\\
	&=a^{\frac{n(n-1)}2}=a^{n\choose 2}
\end{array}
\end{displaymath}
\end{proof}
\end{theorem}
\begin{remark}{Open an (unproved) the combinatorial identity}
$$
\sum_{k=1}^n\left(-\frac12\right)^{n-k}{2j\choose k-1}
\sum_{\nu=0}^{[\frac{n-k}{2}]}{2n-k+1\choose n+1+2\nu}=
2^{n-1}{j-1\choose n-1}
$$
\end{remark}
\section{Conclusion}
The contribution consists of two parts that solve a similar problem in different ways. In the first part, the problem of linear dependence and independence of the functions is solved using the Wronskian. In the second part, we only deal with the special case that was solved in the first part, but unlike the first part, we try to solve it only on the basis of the definition of linear dependence and independence of functions.
Although the first approach is generally used, especially in differential equations, its direct use for the investigation of a sufficiently large number of functions is practically unfeasible. Therefore, it is suitable to first transform the Wronskian appropriately and solve the given problem using the knowledge of differential equations. For Wronskian transformation, properties of determinants of type $\left|\mbf{U}\mbf{V}\mbf{U}^T\right|=\left|\mbf{V}\right|$ are used, where 
$\left|\mbf{U}\right|=1$ and $\mbf{U}$ is a suitably chosen matrix.The main result of the first part of the paper is included in Theorem \ref{wr}, the proof of which was presented in two ways. Theorem \ref{wrf} applies the results of Theorem \ref{wr} to the sequences of derivatives of the functions $x^n\sin x$ and $x^n\cos x$. Theorem \ref{wrfk} generalizes the result of Theorem \ref{wrf}, which is subsequently applied to the sequences of higher derivatives of the functions $x^n\sin x$ and $x^n\cos x$.\par
Solving the second part of the paper leads to non-elementary combinatorial identities, which are generally difficult to express using the sum of hypergeometric functions and are not generally known. The main results of this part of the paper are included in the proof of one combinatorial identity and in the proof of determinants, from which follows the linear independence of the analyzed functions.\par
Also interesting are the results in the appendix where one matrix identity and two non-trivial determinants are proved. Finally, we mention the combinatorial identity, the proof of which is not generally known.

\end{document}